\documentclass[11pt,oneside,reqno]{amsart}
\usepackage{comment}
\usepackage[margin=1in]{geometry}
\usepackage{amsmath,amsthm,amssymb,mathtools,bm}
\usepackage{graphicx}
\usepackage{placeins}
\usepackage{enumitem}
\usepackage{booktabs}
\usepackage{xcolor}
\usepackage{hyperref}
\usepackage[
  backend=biber,
  style=numeric-comp,
  sorting=nyt,
  maxbibnames=99
]{biblatex}
\usepackage{tikz}
\usetikzlibrary{decorations.pathreplacing}

\hypersetup{
  colorlinks=true,
  citecolor=red,
  linkcolor=red,
  urlcolor=black
}

\numberwithin{equation}{section}
\theoremstyle{plain}
\newtheorem{theorem}{Theorem}[section]
\newtheorem{lemma}[theorem]{Lemma}
\newtheorem{corollary}[theorem]{Corollary}
\newtheorem{proposition}[theorem]{Proposition}

\theoremstyle{definition}

\theoremstyle{remark}
\newtheorem{remark}[theorem]{Remark}

\newcommand{\R}{\mathbb R}
\newcommand{\C}{\mathbb C}
\newcommand{\Z}{\mathbb Z}

\newcommand{\T}{\mathbb T}

\newcommand{\1}{\mathbf 1}
\newcommand{\eps}{\varepsilon}

\newcommand{\norm}[1]{\left\|#1\right\|}

\newcommand{\Spec}{\operatorname{Spec}}
\newcommand{\rank}{\operatorname{rank}}
\newcommand{\tr}{\operatorname{tr}}

\newcommand{\ceil}[1]{\left\lceil #1\right\rceil}
\newcommand{\floor}[1]{\left\lfloor #1\right\rfloor}

\title{Displacement structure and  eigenvalue bounds for multiband prolate matrices}
\author{
Luis Gomez}
\address{Department of Mathematics, The University of Texas at Austin, TX, USA}
\email{leg3445@my.utexas.edu}
\author{Hana Huber}
\address{PhD program in mathematics, Graduate Center, CUNY, NY, USA}
\email{hhuber@gradcenter.cuny.edu}
\author{Azita Mayeli}
\address{PhD program in mathematics, Graduate Center, CUNY, NY, USA}
\email{amayeli@gc.cuny.edu}
\thanks{The research of A.M. was supported in part by the National Science Foundation grant DMS-2453769 and the PSC-CUNY grant 67807-00 57.}
\date{October 2026}


\begin{document}

\begin{abstract}
We obtain explicit non-asymptotic bounds for the eigenvalue transition
of multiband prolate matrices. Let $S\subset[-1/2,1/2]$ be a finite
union of pairwise disjoint intervals, and let
$\lambda_0\ge\cdots\ge\lambda_{N-1}$ be the eigenvalues of $B_{N,S}$.
If $\nu$ is the number of boundary points of $S$ on the frequency torus,
then, for every $0<\eps<1/2$,
\[
\#\{k:\eps<\lambda_k<1-\eps\}
\le
\nu\ceil{\frac{\log(4N)}{\pi^2}
\log\!\left(\frac{4}{\eps(1-\eps)}\right)}.
\]
Compared with the multiband estimate of Zhu and
Wakin~\cite{wakin2017approximating}, this improves the dependence on
$\eps$ from order $1/[\eps(1-\eps)]$ to logarithmic dependence.
We also prove a second bound in which the main logarithmic term depends
on the time--bandwidth product $N|S|$ rather than only on $N$.

We also compare $S$ with frequency sets having fewer boundary points.
Using a classical large-sieve estimate, we take into account not only
the total size of the change in the frequency set, but also how that
change is distributed. For sets with many small gaps, this can give a
bound growing like $\log N$ for fixed $\eps$, even when the bounds based
only on the total size of the gaps grow at least linearly in $N$.
We also obtain exponential bounds for individual eigenvalues and use
them to estimate the error in approximating sampled multiband signals
by leading eigenvectors. A discretization argument gives corresponding
results for the continuous problem. Numerical examples compare the
bounds and examine how the locations of the frequency intervals affect
the eigenvalues.
\end{abstract}

\maketitle

\section{Introduction}\label{sec:intro}
Let $S\subset[-1/2,1/2]$ be a measurable frequency set and let $B_S$
denote the orthogonal projection in $\ell^2(\Z)$ onto sequences whose
discrete-time Fourier transform is supported in $S$. We identify the
endpoints $-1/2$ and $1/2$, so the frequency variable is naturally
viewed on the torus $\T=\R/\Z$. Let $T_N$ be the projection in
$\ell^2(\Z)$ onto the time indices $I_N:=\{0,1,\dots,N-1\}$.
The finite time--frequency limiting operator on $\operatorname{ran}(T_N)$  is
\[
B_{N,S}:=T_NB_ST_N.
\]
Identifying $\operatorname{ran}(T_N)$ with $\C^N$, we regard $B_{N,S}$ as
the $N\times N$ matrix with entries
\[
B_{N,S}[m,n]
=
\int_S e^{i2\pi\xi(m-n)}\,d\xi,
\qquad 0\le m,n\le N-1.
\]
When $|W|<1/2$ and  $S=[-W,W]$, $B_{N,S}$ is the classical prolate matrix,
whose eigenvectors are the Slepian basis vectors. Its eigenvalues
cluster near $1$ and $0$, with a small number in between; see the
classical work of Slepian, Landau, Pollak, and Widom~\cite{Bell1,Bell2,Bell3,Bell5,LandauWidom80}.
Multiband concentration spaces arise in the approximation of sampled
signals \cite{Davenport2012,wakin2017approximating,Li2025Approximating}.
The eigenvalues determine how many leading eigenvectors are needed
for a prescribed accuracy. Section~\ref{sec:signal-approximation}
gives the corresponding dimension and error bounds.

Karnik, Romberg, and Davenport~\cite{Karnik21} obtained explicit
one-band eigenvalue estimates using the low displacement rank of
finite matrix truncations of $(I-T_N)B_ST_N$, which records the
bandlimited output outside the time window,  together with
singular-value bounds for matrices with displacement structure due to
Beckermann and Townsend~\cite{BeckermannTownsend2019}. 
Their earlier work
\cite[Sections~III and~VI]{Karnik2019Nonuniform}
showed that the relevant multiband sampling matrices are generalized
Cauchy or generalized Cauchy-like matrices.  
We apply the same displacement singular-value method to the
representation in terms of the interval endpoints and verify its exact rank.
The resulting estimates improve the dependence on $\varepsilon$ in the
trace-defect bound of Zhu and Wakin~\cite{wakin2017approximating}.

Throughout the discrete part, let $S\subset\T$ be a finite
union of frequency intervals with $0<|S|<1$.
An interval on $\T$ may cross the identified endpoints
$-1/2$ and $1/2$. Since $S$ has a gap, a frequency translation
modulo $1$ places this identified point inside a gap without
changing the eigenvalues (Lemma~\ref{lem:frequency-translation}).
We may therefore write
\begin{equation}
S=\bigcup_{r=1}^{M}[\alpha_r,\beta_r],
\qquad
-\frac12<\alpha_1<\beta_1<\cdots<\alpha_M<\beta_M<\frac12.
\label{eq:general-components}
\end{equation}
We write $\partial_{\T}S$ for the boundary in $\T$ and $\#$ for
cardinality, and set $\nu:=\#\partial_{\T}S=2M$.
For the refined estimate it is convenient to write the same components as
\begin{equation}
S=\bigcup_{r=1}^{M}[c_r-W_r,c_r+W_r],
\label{eq:component-form}
\end{equation}
where $c_r:=(\alpha_r+\beta_r)/2$ and
$W_r:=(\beta_r-\alpha_r)/2>0$. For
Corollary~\ref{cor:symmetric-annuli} and the numerical examples,
we also consider symmetric sets of the form
\begin{equation}
S=\bigcup_{j=1}^{J}\bigl([-b_j,-a_j]\cup[a_j,b_j]\bigr),
\qquad
0\le a_1<b_1<a_2<b_2<\cdots<a_J<b_J \le \tfrac12.
\label{eq:annular-form}
\end{equation}
The main theorems do
not require symmetry.

The matrix $B_{N,S}$ is a positive contraction. We denote its
eigenvalues, counted with multiplicity, by
$1>\lambda_0\ge\lambda_1\ge\cdots\ge\lambda_{N-1}>0$. (When it is useful to display the dependence on the parameters, we write \(\lambda_k=\lambda_k(N,S)\).)  
Moreover, $\tr(B_{N,S})=N|S|$, where $\tr$ denotes the trace of a matrix.
The quantity $N|S|$ is the time--bandwidth product, also known as the
Shannon number, and serves as the natural reference index for the
eigenvalue transition (see Proposition~\ref{prop:crossing}).
Our main goal is to control the size of the plunge region
\[
M_\eps(N,S):=\#\{k:\eps<\lambda_k(N,S)<1-\eps\},
\qquad 0<\eps<\tfrac12.
\]

\subsection{Transition bounds}
The two transition estimates below follow from the endpoint structure
of finite truncations $X_L$ of $(I-T_N)B_ST_N$.
The matrix $X_L$ is defined in Section~\ref{sec:truncation};
$C_L$ and $D$ are the diagonal matrices of
Section~\ref{sec:displacement}. We first record the rank calculation.

\begin{theorem}[Endpoint displacement rank]
\label{thm:displacement-main}
Let $S$ be a union of $M$ connected frequency intervals as in \eqref{eq:general-components}, and set $\nu=\#\partial_{\T}S=2M$.  Then, for every $L\ge1$,
\begin{equation}
\rank(C_LX_L-X_LD)\le \nu.
\label{eq:rank-upper-main}
\end{equation}
If $N,L\ge\nu$, then
\begin{equation}
\rank(C_LX_L-X_LD)=\nu.
\label{eq:rank-exact-main}
\end{equation}
\end{theorem}

For one interval, this is the rank-two structure used
in~\cite{Karnik21}. Applying their displacement singular-value
estimate gives the first transition bound.

\begin{theorem}[Transition bound depending on $N$]
\label{thm:first-main}
Let $S$ be a union of $M$ connected frequency intervals and let
$\nu=\#\partial_{\T}S=2M$. Let
$1>\lambda_0\ge\lambda_1\ge\cdots\ge\lambda_{N-1}>0$
be the eigenvalues of $B_{N,S}$. Then, for every $\eps\in(0,1/2)$,
\begin{equation}
M_\eps(N,S)
\le
\nu\ceil{\frac{\log(4N)}{\pi^2}
\log\!\left(\frac{4}{\eps(1-\eps)}\right)}.
\label{eq:first-main}
\end{equation}
\end{theorem}

For a single interval, $M=1$ and $\nu=2$, so \eqref{eq:first-main} is exactly the first bound in \cite{Karnik21}.  The next corollary states the symmetric annular case explicitly.

\begin{corollary}[Symmetric annular sets]
\label{cor:symmetric-annuli}
Let $S$ be given by \eqref{eq:annular-form}.  If $b_J<1/2$, then
\[
\nu=
\begin{cases}
4J,&a_1>0,\\
4J-2,&a_1=0.
\end{cases}
\]
Hence Theorem~\ref{thm:first-main} holds with this value of $\nu$, and the displacement rank equals $\nu$ whenever $N,L\ge\nu$.  If $b_J=1/2$, the two outer endpoints are identified on $\T$ and the boundary count is smaller by $2$.
\end{corollary}
The second estimate replaces the main $\log N$ term by a
logarithm involving $N|S|$, at the expense of an additional term
depending on $M$ and $\eps$.

\begin{theorem}[Transition bound depending on $N|S|$]
\label{thm:second-main}
Let $S$ be a union of $M$ connected frequency intervals and suppose that $0<|S|\leq \frac12$.
Let $0<\eps<1/2$, define $q_\eps:=\eps(1-\eps)$, and set
\begin{equation}
 k_0:=
\ceil{
\frac{\log(8N|S|+4)}{\pi^2}
\log\!\left(\frac{5}{q_\eps}\right)},
\label{eq:k0-main}
\end{equation}
and
\begin{equation}
 k_1:=
\ceil{
\frac{1}{2\log(48/\pi)}
\log\!\left(\frac{56000M^2}{\pi q_\eps}\right)}.
\label{eq:k1-main}
\end{equation}
Then
\begin{equation}
M_\eps(N,S)\le 2M(k_0+k_1)=\nu(k_0+k_1).
\label{eq:second-main}
\end{equation}
In particular,
\begin{align}
M_\eps(N,S)
&\le
\frac{2M}{\pi^2}\log(8N|S|+4)
\log\!\left(\frac{5}{q_\eps}\right)
\notag\\
&\quad+
\frac{M}{\log(48/\pi)}
\log\!\left(\frac{56000M^2}{\pi q_\eps}\right)
+4M.
\label{eq:second-main-simple}
\end{align}
\end{theorem}

For fixed $M$, the right-hand side of \eqref{eq:second-main-simple} is
$O\!\left(M\log(N|S|+1)\log\frac1\eps\right)$
with an additional term of size $O(M\log(M/\eps))$.  The first logarithmic factor is the multiband analogue of the time--bandwidth term in \cite{Karnik21}.

\subsection{Comparison with simpler frequency sets}
The rank calculation counts every endpoint. To obtain smaller bounds
when $S$ has many small gaps, we compare it with a set having fewer
boundary points. For $0<t<1/2$, set
\begin{equation}
\Phi_N(t)
:=
\ceil{
\frac{\log(4N)}{\pi^2}
\log\!\left(\frac{4}{t(1-t)}\right)
},
\label{eq:PhiN}
\end{equation}
and for $\rho\ge0$ define the effective boundary count
\begin{equation}
\nu_{\mathrm{eff}}(S;\rho)
:=
\min\Bigl\{
\#\partial_{\T}S_0:
S_0\subset\T\text{ is a finite union of intervals and }
|S\triangle S_0|\le\rho
\Bigr\}.
\label{eq:effective-boundary}
\end{equation}
We allow $S_0=\varnothing$ and $S_0=\T$, both with boundary count zero.
Since $S$ itself is admissible, the minimum in
\eqref{eq:effective-boundary} is well defined.
We also consider how the modification of a frequency set is
distributed. For a measurable set $E\subset\T$, define
\begin{equation}
\mathcal D_N(E)
:=
N\sup_{x\in\T}
\bigl|E\cap[x,x+N^{-1}]\bigr|,
\qquad
\omega_N(E)
:=
\min\{N|E|,\,2\mathcal D_N(E)\}.
\label{eq:local-density}
\end{equation}
Intervals are interpreted on the frequency torus.
The quantity $\mathcal D_N(E)$ measures the largest proportion
of an interval of length $N^{-1}$ occupied by $E$.
Related local density estimates appear in the classical
large-sieve work of Donoho and Logan~\cite{DonohoLogan1992}.

For $\eta\ge0$, define
\begin{equation}
\nu_{\mathrm{loc}}(N,S;\eta)
:=
\min\Bigl\{
\#\partial_{\T}S_0:
S_0\subset\T\text{ is a finite union of intervals},
\quad
\omega_N(S\triangle S_0)\le\eta
\Bigr\}.
\label{eq:local-effective-boundary}
\end{equation}
As before, we allow $S_0=\varnothing$ and $S_0=\T$.
Since $S_0=S$ is admissible, the minimum is well defined.

The effective boundary count $\nu_{\mathrm{eff}}(S;\rho)$
depends on the total measure of the change in the frequency
set, while $\nu_{\mathrm{loc}}(N,S;\eta)$ also takes into
account how that change is distributed at scale $N^{-1}$.

\begin{theorem}[Effective boundary bound]
\label{thm:effective-boundary}
Let $S\subset\T$ be a finite union of intervals with $0<|S|<1$, and let
$0<\eps<1/2$.  For every $\rho\ge0$ and every $0<\tau<\eps$,
\begin{equation}
M_\eps(N,S)
\le
\nu_{\mathrm{eff}}(S;\rho)\,\Phi_N(\tau)
+
\frac{N\rho}{\eps-\tau}.
\label{eq:effective-boundary-additive}
\end{equation}
In particular, taking $\tau=\eps/2$ gives
\begin{equation}
M_\eps(N,S)
\le
\nu_{\mathrm{eff}}(S;\rho)\,\Phi_N(\eps/2)
+
\frac{2N\rho}{\eps}.
\label{eq:effective-boundary-half}
\end{equation}
If, in addition, $N\rho<\eps$,  then we also have
\begin{equation}
M_\eps(N,S)
\le
\nu_{\mathrm{eff}}(S;\rho)\,\Phi_N(\eps-N\rho).
\label{eq:effective-boundary-shifted}
\end{equation}

More generally, for every $0\le\eta<\eps$,
\begin{equation}
M_\eps(N,S)
\le
\nu_{\mathrm{loc}}(N,S;\eta)
\Phi_N(\eps-\eta).
\label{eq:effective-boundary-local}
\end{equation}
Moreover, for every $\rho\ge0$,
\begin{equation}
\nu_{\mathrm{loc}}(N,S;N\rho)
\le
\nu_{\mathrm{eff}}(S;\rho).
\label{eq:local-global-boundary-comparison}
\end{equation}
\end{theorem}

The estimate based only on the total measure of the change may be optimized over $\rho$ and $\tau$.
The local estimate also uses the distribution of the change on
intervals of length $N^{-1}$. The following consequence treats gaps
inside a single interval.

\begin{corollary}[Small gaps]
\label{cor:small-gaps}
Let $J\subset\T$ be an interval with $0<|J|<1$, and let $G\subset J$ be a finite union
of pairwise disjoint intervals.  Set $S=J\setminus G$, assume $0<|S|<1$, and
let $g:=|G|$.  Then, for every $0<\eps<1/2$,
\begin{equation}
M_\eps(N,S)
\le
2\Phi_N(\eps/2)+\frac{2Ng}{\eps}.
\label{eq:small-gap-additive}
\end{equation}
If $Ng<\eps$, then
\begin{equation}
M_\eps(N,S)
\le
2\Phi_N(\eps-Ng).
\label{eq:small-gap-shifted}
\end{equation}
In particular, if $Ng\le\eps/2$, then
\begin{equation}
M_\eps(N,S)
\le
2\Phi_N(\eps/2).
\label{eq:small-gap-half}
\end{equation}
More generally, put $\gamma:=\omega_N(G)$.
If $\gamma<\eps$, then
\begin{equation}
M_\eps(N,S)
\le
2\Phi_N(\eps-\gamma).
\label{eq:small-gap-local}
\end{equation}
In particular, the bound
\eqref{eq:small-gap-half}
also holds whenever
$\mathcal D_N(G)\le\frac{\eps}{4}$,
without requiring $Ng\le\eps/2$.

If $G$ consists of $m$ gaps of equal length $\delta$, then $g=m\delta$.
Thus \eqref{eq:small-gap-additive} depends on the gaps through their total
length, while \eqref{eq:small-gap-local} also takes into account how the
gaps are distributed.
\end{corollary}

The remaining results describe consequences for individual eigenvalues,
signal approximation, and continuous limiting operators.
Proposition~\ref{prop:crossing}, taken from
Zhu and Wakin~\cite[Theorem~3.3]{wakin2017approximating}, locates the
$1/2$ crossing within $O(\nu)$ indices of $N|S|$.
Combining this with Theorem~\ref{thm:first-main} gives exponential
bounds for $1-\lambda_k$ before the transition and for $\lambda_k$
after it (Corollary~\ref{cor:eigenvalue-decay}).
These imply uniform approximation-error bounds for sampled
multiband signals (Corollary~\ref{cor:signal-exponential}).
Section~\ref{sec:continuous} transfers the discrete estimates to the
continuous problem by discretization.

\subsection{Relation to previous work}
Theorem~\ref{thm:first-main} recovers the first one-band estimate of Karnik, Romberg, and Davenport~\cite{Karnik21}  when \(M=1\), while Theorem~\ref{thm:second-main} is a multiband analogue of their time–bandwidth-dependent estimate.
 The multiband Cauchy-like  representation    comes 
from~\cite[Sections~III and~VI]{Karnik2019Nonuniform}; the present
argument applies it to samples outside the time window and verifies
when the displacement rank equals the endpoint count.

Zhu and Wakin~\cite{wakin2017approximating} bounded the trace defect
$\sum_k\lambda_k(1-\lambda_k)$, obtaining a multiband transition
estimate with threshold dependence $1/[\eps(1-\eps)]$.
Theorem~\ref{thm:first-main} replaces this by
$\log(1/[\eps(1-\eps)])$.
Their Theorem~3.4 also gives exponential concentration for fixed
frequency sets and sufficiently large $N$, with set-dependent
constants.  Corollary~\ref{cor:eigenvalue-decay} instead gives an explicit estimate
valid for every $N$, with explicit dependence on $\nu$ and on the
distance from the $1/2$ crossing, using the argument of
\cite[Corollary~1]{Karnik21}. 

The comparison with simpler frequency sets combines these transition
bounds with standard matrix perturbation inequalities and the classical
large-sieve method of Donoho and Logan~\cite{DonohoLogan1992}.
Proposition~\ref{prop:local-global-separation} shows that the local
estimate can grow logarithmically  whereas the right-hand sides of the
corresponding total-measure bounds are of order at least $N$.

Continuous eigenvalue-distribution results include Landau and
Widom~\cite{LandauWidom80} and, for multiband localization,
Izu and Lakey~\cite{IzuLakey2009}. Recent non-asymptotic estimates
of Kulikov and Dam Larsen~\cite{Kulikov-Larsen} and
Azimifard~\cite{Azimi1} have different parameter ranges   and depend
differently on the geometry of the frequency set. The continuous results here are consequences of our
discrete estimates, not a separate method. The recent preprint
\cite{Azimifard2026MeasureSensitive}  proves related bounds for continuous localization problems, including estimates that depend on changes in the frequency set.  Our theorem, Theorem~\ref{thm:effective-boundary}, on the other hand, compares two discrete frequency sets \(S\) and \(S_0\) and also takes into account how the change is distributed over intervals of length \(N^{-1}\).

Higher-dimensional results under geometric assumptions appear
in~\cite{ArieAzita23,marceca2023,hughes2024eigenvalue}.
The trace-defect estimates of
\cite{HIM-trace-rough,Mayeli2026TraceDefect} treat more general,
including rough, geometries, but give algebraic threshold dependence.
Here the finite endpoint representation supplies the additional
singular-value decay. For further one-dimensional estimates see
\cite{Israel15}; for Cartesian sets and tensor-product
methods see~\cite{GJMP_prolateHD}.

\subsection{Organization of the paper}
Sections~\ref{sec:prelim}--\ref{sec:truncation} give the operator
identities used in the proofs. Sections~\ref{sec:displacement}
and~\ref{sec:refined} prove the two transition estimates, and
Section~\ref{sec:stability} proves the comparison results.
Sections~\ref{sec:eigenvalue-decay}--\ref{sec:continuous} give the
eigenvalue, approximation, and continuous consequences.
Section~\ref{sec:numerics} presents selected numerical examples.
Appendix~\ref{app:chebyshev} contains the scalar approximation estimate.

\section{Preliminaries and notation}
\label{sec:prelim}

Let $x\in\ell^2(\Z)$.  We use the discrete-time Fourier transform
\[
\widehat{x}(\xi)=\sum_{n\in\Z}x[n]e^{-i2\pi\xi n},
\qquad \xi\in[-1/2,1/2],
\]
with inverse
\[
x[n]=\int_{-1/2}^{1/2}\widehat{x}(\xi)e^{i2\pi\xi n}\,d\xi.
\]
With this normalization, the Fourier transform is unitary and
\[
\sum_{n\in\Z}x[n]\overline{y[n]}
=
\int_{-1/2}^{1/2}\widehat{x}(\xi)\overline{\widehat{y}(\xi)}\,d\xi.
\]

For a measurable set $S\subset[-1/2,1/2]$, let $B_S$ be the
bandlimiting projection defined by
$\widehat{B_Sx}(\xi)=\1_S(\xi)\widehat{x}(\xi)$.
Its kernel is
\begin{equation}
K_S(\tau):=\int_S e^{i2\pi\xi\tau}\,d\xi,
\qquad \tau\in\R.
\label{eq:KS-def}
\end{equation}

Thus
\begin{equation}
(B_Sx)[n]=\sum_{m\in\Z}K_S(n-m)x[m].
\label{eq:BS-convolution}
\end{equation}

With this notation, the multiband prolate matrix introduced in
Section~\ref{sec:intro} is
$B_{N,S}=\bigl(K_S(m-n)\bigr)_{m,n\in I_N}$.

\begin{lemma}[Basic properties of the kernel]
\label{lem:kernel-basic}
Let $S\subset[-1/2,1/2]$ be measurable.
Then
\begin{equation}
K_S(-\tau)=\overline{K_S(\tau)},
\label{eq:kernel-hermitian}
\end{equation}
and
\begin{equation}
\sum_{\tau\in\Z}|K_S(\tau)|^2=|S|.
\label{eq:kernel-parseval}
\end{equation}
If $S=-S$, then $K_S$ is real-valued and even.
\end{lemma}

The projection identity $B_S^2=B_S$ gives the following kernel identity.

\begin{lemma}[Kernel identity]
\label{lem:reproducing}
For all $m,n\in\Z$,
\begin{equation}
\sum_{\ell\in\Z}
\overline{K_S(\ell-m)}K_S(\ell-n)
=K_S(m-n).
\label{eq:reproducing-general}
\end{equation}
If $S=-S$, then
\begin{equation}
\sum_{\ell\in\Z}
K_S(\ell-m)K_S(\ell-n)
=K_S(m-n).
\label{eq:reproducing-symmetric}
\end{equation}
\end{lemma}

\begin{proof}
The projection identity $B_S^2=B_S$ and
\eqref{eq:kernel-hermitian} give
\[
K_S(m-n)
=\sum_{\ell\in\Z}K_S(m-\ell)K_S(\ell-n)
=\sum_{\ell\in\Z}\overline{K_S(\ell-m)}K_S(\ell-n).
\]
For $S=-S$, the kernel is real by Lemma~\ref{lem:kernel-basic}.
\end{proof}

For $N\ge1$, let $T_N$ be the time-limiting projection
\[
(T_Nx)[n]=
\begin{cases}
x[n],&n\in I_N,\\
0,&n\notin I_N.
\end{cases}
\]
We identify the operator $T_NB_ST_N$ with the matrix
$B:=B_{N,S}=\bigl(K_S(m-n)\bigr)_{m,n\in I_N}$.

\begin{lemma}[Frequency translation invariance]
\label{lem:frequency-translation}
Let $\theta\in\R$ and let $S_\theta=(S+\theta)\bmod 1$.
Define the diagonal matrix $D_\theta$ by
$D_\theta[n,n]=e^{i2\pi\theta n}$ for $n\in I_N$.
Then
\begin{equation}
B_{N,S_\theta}=D_\theta B_{N,S}D_\theta^*.
\label{eq:frequency-translation}
\end{equation}
In particular, frequency translation does not change the eigenvalues of $B_{N,S}$.
\end{lemma}

\begin{proof}
For every integer $t$, the exponential $e^{i2\pi\xi t}$ is $1$-periodic in $\xi$, and a change of variables modulo $1$ gives
$K_{S_\theta}(t)=e^{i2\pi\theta t}K_S(t)$.
Substituting $t=m-n$ gives \eqref{eq:frequency-translation}.
\end{proof}

\begin{proposition}[Spectrum and trace]
\label{prop:spectrum-trace}
Assume $0<|S|<1$. Then $B$ is Hermitian and $0\prec B\prec I$.
In particular, its eigenvalues can be ordered as
$1>\lambda_0\ge\lambda_1\ge\cdots\ge\lambda_{N-1}>0$.
Moreover,
\begin{equation}
\sum_{k=0}^{N-1}\lambda_k=N|S|.
\label{eq:trace}
\end{equation}
\end{proposition}

\begin{proof}
For $v\in\C^N$, extend $v$ by zero outside $I_N$.  Then
\[
\langle Bv,v\rangle
=
\int_S\left|\sum_{n=0}^{N-1}v[n]e^{-i2\pi\xi n}\right|^2d\xi.
\]
It follows that $0\preceq B\preceq I$.  If $v\ne0$, the trigonometric polynomial inside the integral cannot vanish on a set of positive measure.  Since both $S$ and its complement have positive measure, the two inequalities are strict.  Finally, $B[n,n]=K_S(0)=|S|$, so $\tr(B)=N|S|$.
\end{proof}

Throughout the paper, $\|A\|$ denotes the operator norm. For matrices,
$\|A\|_F$ denotes the Frobenius norm, and for integral operators,
$\|A\|_{\mathrm{HS}}$ denotes the Hilbert--Schmidt norm.
For a matrix $A\in\C^{p\times q}$, its singular values are indexed
as $\sigma_1(A)\ge\cdots\ge\sigma_{\min\{p,q\}}(A)\ge0$.
We set $\sigma_j(A)=0$ for $j>\min\{p,q\}$.

\section{Samples outside the time window and the plunge region}
\label{sec:truncation}

The projection identity $B_S^2=B_S$ allows us to write the entries of   $B-B^2$ in terms of samples outside the time window.

\begin{lemma}[Representation of $B-B^2$]
\label{lem:BB2}
For $m,n\in I_N$,
\begin{equation}
(B-B^2)[m,n]
=
\sum_{\ell\notin I_N}
\overline{K_S(\ell-m)}K_S(\ell-n).
\label{eq:BB2-outside}
\end{equation}
\end{lemma}

\begin{proof}
Subtract the sum over $\ell\in I_N$ from
\eqref{eq:reproducing-general}, using
\[
B^2[m,n]=\sum_{\ell\in I_N}
\overline{K_S(\ell-m)}K_S(\ell-n).
\]
\end{proof}
For $L\ge1$, define
\[
I_L^{\mathrm{out}}
:=\{-L,\dots,-1\}\cup\{N,\dots,N+L-1\}
\]
and define $X_L\in\C^{2L\times N}$ by
\begin{equation}
X_L[\ell,n]=K_S(\ell-n),
\qquad \ell\in I_L^{\mathrm{out}},\quad n\in I_N.
\label{eq:XL-def}
\end{equation}
The rows are indexed by $I_L^{\mathrm{out}}$.

\begin{proposition}[Truncation]
\label{prop:truncation}
As $L\to\infty$,
\begin{equation}
X_L^*X_L\longrightarrow B-B^2
\qquad\text{in Frobenius norm}.
\label{eq:XL-convergence}
\end{equation}
Moreover,
\begin{equation}
0\preceq X_L^*X_L\preceq B-B^2\preceq \frac14 I,
\label{eq:XL-order}
\end{equation}
and
\begin{equation}
\norm{X_L}\le\frac12.
\label{eq:XL-norm}
\end{equation}
\end{proposition}

\begin{proof}
The entries of $X_L^*X_L$ are the sums in
\eqref{eq:BB2-outside} restricted to $I_L^{\mathrm{out}}$.
Cauchy--Schwarz and \eqref{eq:kernel-parseval} give
\[
\sum_{\ell\in\Z}|K_S(\ell-m)K_S(\ell-n)|\le|S|,
\]
so these sums converge entrywise, and hence in Frobenius norm for
fixed $N$.
The omitted rows contribute a positive semidefinite matrix, which gives
$0\preceq X_L^*X_L\preceq B-B^2$.
Since the eigenvalues of $B-B^2$ are $\lambda_k(1-\lambda_k)\le1/4$,
\eqref{eq:XL-order} and \eqref{eq:XL-norm} follow.
\end{proof}

We use the following elementary lemma, which is essentially \cite[Lemma~1]{Karnik21}.  We include the short proof for completeness.

\begin{lemma}[Transition eigenvalues]
\label{lem:transition}
Let $B$ be Hermitian with $0\preceq B\preceq I$, \(0<\varepsilon<1/2\),  and let  
$1\ge\lambda_0(B)\ge\lambda_1(B)\ge\cdots\ge\lambda_{N-1}(B)\ge0$ be the   eigenvalues of $B$.
Suppose that $X_L^*X_L\to B-B^2$ in operator norm. If, for some
integer $r\ge0$ and all sufficiently large $L$,
$\sigma_{r+1}(X_L)\le\sqrt{\eps(1-\eps)}$, then
$\#\{k:\eps<\lambda_k(B)<1-\eps\}\le r$.
\end{lemma}

\begin{proof}
The eigenvalues of $B-B^2$ are $\lambda_k(B)(1-\lambda_k(B))$.  Thus
\[
\eps<\lambda_k(B)<1-\eps
\quad\Longleftrightarrow\quad
\lambda_k(B)(1-\lambda_k(B))>\eps(1-\eps).
\]
The eigenvalues of $X_L^*X_L$ are the squared singular values of $X_L$.  The result follows from $X_L^*X_L\to B-B^2$ in operator norm and Weyl's eigenvalue inequality.
\end{proof}

\section{Displacement structure and a bound on the plunge region}
\label{sec:displacement}

We now show that the displacement rank of the matrix $X_L$, formed from
samples outside the time window $I_N$, is controlled by the number of frequency
boundary points, and is equal to this number when the matrix dimensions are
large enough.
The endpoint expansion below is the specialization
to integer samples of  the generalized Cauchy-like representation for multiband sampling
matrices  in 
\cite[Sections~III and~VI]{Karnik2019Nonuniform}.
We include the calculation to fix the notation
and to verify the exact-rank assertion in
Theorem~\ref{thm:displacement-main}.

Write $S$ as in \eqref{eq:general-components}. For $\tau\ne0$,
\begin{equation}
K_S(\tau)
=
\frac{1}{i2\pi\tau}
\sum_{r=1}^{M}
\left(e^{i2\pi\beta_r\tau}-e^{i2\pi\alpha_r\tau}\right).
\label{eq:endpoint-kernel}
\end{equation}
Define diagonal matrices
\begin{equation}
D[n,n]=n,
\qquad n\in I_N,
\label{eq:D-def}
\end{equation}
and
\begin{equation}
C_L[\ell,\ell]=\ell,
\qquad \ell\in I_L^{\mathrm{out}}.
\label{eq:C-def}
\end{equation}

\begin{proof}[Proof of Theorem~\ref{thm:displacement-main}]
Since $\ell\notin I_N$ and $n\in I_N$, we have $\ell\ne n$.  By \eqref{eq:endpoint-kernel},
\begin{align}
(C_LX_L-X_LD)[\ell,n]
&=(\ell-n)K_S(\ell-n)\notag\\
&=\frac{1}{i2\pi}
\sum_{r=1}^{M}
\left(
 e^{i2\pi\beta_r\ell}e^{-i2\pi\beta_r n}
-
 e^{i2\pi\alpha_r\ell}e^{-i2\pi\alpha_r n}
\right).
\label{eq:endpoint-displacement}
\end{align}
This is a sum of $2M=\nu$ rank-one matrices, so \eqref{eq:rank-upper-main} follows.

For the equality statement, list the endpoints as $\omega_1,\dots,\omega_\nu$ and set $z_j=e^{i2\pi\omega_j}$.  By our choice of representation of $S$, the nodes $z_j$ are distinct.
The column factors in \eqref{eq:endpoint-displacement} form a Vandermonde matrix with rows indexed by $0,\dots,N-1$, and hence have full column rank when $N\ge\nu$.
Restricting the row factors to $-\nu,\dots,-1$ gives, up to multiplication by an invertible diagonal matrix, another Vandermonde matrix with the same distinct nodes, so the row factor has full column rank when $L\ge\nu$.
The representation in \eqref{eq:endpoint-displacement} has the form $UAV^*$, where $U$ and $V$ are these two Vandermonde factors and $A$ is an invertible diagonal matrix.
Since $U$ and $V$ have full column rank, $UAV^*$ has rank $\nu$.  This proves \eqref{eq:rank-exact-main}.
\end{proof}

\begin{remark}[Boundary points  and interval representations]
\label{rem:annuli-rank}
Equation \eqref{eq:endpoint-displacement} shows why the boundary, rather than a chosen interval decomposition, is the natural parameter.  If adjacent intervals are written with a shared endpoint, the two endpoint contributions cancel.  Likewise, the points $-1/2$ and $1/2$ represent the same point of $\T$; if both occur as endpoints,  their contributions cancel after the corresponding intervals are
viewed as a single component on $\mathbb T$.   By Lemma~\ref{lem:frequency-translation}, we can translate the
frequency set and avoid this special case.
\end{remark}

We use the following result from \cite[Theorem~4]{Karnik21}, based on the displacement estimates of Beckermann and Townsend \cite{BeckermannTownsend2019}.

\begin{theorem}[Displacement singular-value bound]
\label{thm:BT}
Suppose $X\in\C^{P\times Q}$ satisfies
$CX-XD=UV^*$,
where $C$ and $D$ are normal with real spectra and $U,V$ have $\nu$ columns.  Assume
\[
\Spec(C)\subset(-\infty,c_1]\cup[c_2,\infty),
\qquad
\Spec(D)\subset[d_1,d_2],
\]
with $c_1<d_1<d_2<c_2$.  Then, for every integer $k\ge0$,
\begin{equation}
\sigma_{\nu k+1}(X)
\le
4\norm{X}
\exp\!\left(-\frac{\pi^2k}{\log(16\gamma)}\right),
\qquad
\gamma=\frac{(c_2-d_1)(d_2-c_1)}{(c_2-d_2)(d_1-c_1)}.
\label{eq:BT-bound}
\end{equation}
\end{theorem}

\begin{proof}[Proof of Theorem~\ref{thm:first-main}]
The case $N=1$ is immediate, since $M_\eps(1,S)\le1$ while the
right-hand side of \eqref{eq:first-main} is at least $\nu\ge2$.
Assume $N\ge2$. Apply Theorem~\ref{thm:BT} to $X_L$ with
$c_1=-1$, $d_1=0$, $d_2=N-1$, and $c_2=N$.
Then $\gamma=N^2$. By Proposition~\ref{prop:truncation}, $\norm{X_L}\le1/2$, so
\begin{equation}
\sigma_{\nu k+1}(X_L)
\le
2\exp\!\left(-\frac{\pi^2k}{2\log(4N)}\right),
\qquad k\ge0.
\label{eq:first-sv}
\end{equation}
Set
\[
k=
\ceil{\frac{\log(4N)}{\pi^2}
\log\!\left(\frac{4}{\eps(1-\eps)}\right)}.
\]
With this choice, \eqref{eq:first-sv} gives
$\sigma_{\nu k+1}(X_L)\le\sqrt{\eps(1-\eps)}$.

The inequality $M_\eps(N,S)\le\nu k$ follows from Lemma~\ref{lem:transition}, and the definition of $k$ gives \eqref{eq:first-main}.
\end{proof}

\section{A refined bound in terms of the time--bandwidth product}
\label{sec:refined}

We prove Theorem~\ref{thm:second-main} using the three-block decomposition
from \cite{Karnik21}. We split the rows of $X_L$ according to their distance
from the time window $I_N$. This gives one far block and two near blocks, as
shown in Figure~\ref{fig:three-block}. The far block is controlled by the
endpoint displacement rank, while each near block is written as a sum of $M$
  single-interval matrices  multiplied by diagonal phase factors.

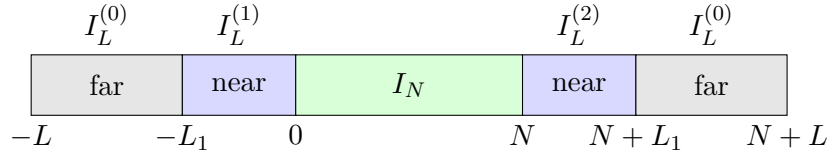
\begin{figure}[!htbp]
\centering
\begin{tikzpicture}[x=1cm,y=1cm]

\draw[fill=gray!20]  (-5,0) rectangle (-3,0.8);
\draw[fill=blue!15]  (-3,0) rectangle (-1.5,0.8);
\draw[fill=green!15] (-1.5,0) rectangle (1.5,0.8);
\draw[fill=blue!15]  (1.5,0) rectangle (3,0.8);
\draw[fill=gray!20]  (3,0) rectangle (5,0.8);

\node at (-4,0.4) {far};
\node at (-2.25,0.4) {near};
\node at (0,0.4) {$I_N$};
\node at (2.25,0.4) {near};
\node at (4,0.4) {far};

\node[below] at (-5,0) {$-L$};
\node[below] at (-3,0) {$-L_1$};
\node[below] at (-1.5,0) {$0$};
\node[below] at (1.5,0) {$N$};
\node[below] at (3,0) {$N+L_1$};
\node[below] at (5,0) {$N+L$};

\node[above] at (-4,0.8) {$I_L^{(0)}$};
\node[above] at (-2.25,0.8) {$I_L^{(1)}$};
\node[above] at (2.25,0.8) {$I_L^{(2)}$};
\node[above] at (4,0.8) {$I_L^{(0)}$};

\end{tikzpicture}

\caption{Three-block decomposition of the outside index set used in
the proof of Theorem~\ref{thm:second-main}.
The two near blocks lie immediately outside the time window $I_N$,
while $I_L^{(0)}$ consists of the remaining outside indices.}
\label{fig:three-block}
\end{figure}

Assume $0<|S|\leq 1/2$ and set
\begin{equation}
L_1:=\floor{\frac{1}{2|S|}}.
\label{eq:L1}
\end{equation}
Then $L_1\ge1$.  For $L\ge L_1+1$, split $I_L^{\mathrm{out}}$ into
\begin{align*}
I_L^{(0)}
&:=\{-L,\dots,-L_1-1\}\cup\{N+L_1,\dots,N+L-1\},\\
I_L^{(1)}
&:=\{-L_1,\dots,-1\},\\
I_L^{(2)}
&:=\{N,\dots,N+L_1-1\}.
\end{align*}
Let $X_L^{(i)}$ be the corresponding row submatrices.  Then
\begin{equation}
X_L^*X_L
=X_L^{(0)*}X_L^{(0)}
+X_L^{(1)*}X_L^{(1)}
+X_L^{(2)*}X_L^{(2)}.
\label{eq:three-block-gram}
\end{equation}

\subsection{The far blocks}

The matrix $X_L^{(0)}$ satisfies the same displacement-rank bound $\nu=2M$. Its row indices are separated from $I_N$ by the distance $L_1+1$.

\begin{lemma}[Far blocks]
\label{lem:far-block}
Assume $N\ge2$. For every $k\ge0$,
\begin{equation}
\sigma_{\nu k+1}(X_L^{(0)})
\le
2\exp\!\left(
-\frac{\pi^2k}{2\log(8N|S|+4)}
\right).
\label{eq:far-block}
\end{equation}
\end{lemma}

\begin{proof}
We apply Theorem~\ref{thm:BT} with $c_1=-L_1-1$, $d_1=0$,
$d_2=N-1$, and $c_2=N+L_1$.
Then
$\gamma=\left(\frac{N+L_1}{L_1+1}\right)^2$.
Since $L_1\ge (2|S|)^{-1}-1$ and the function $x\mapsto(N+x)/(x+1)$ is nonincreasing for $N\ge1$,
\[
\gamma
\le
(2N|S|+1-2|S|)^2
\le
(2N|S|+1)^2.
\]
Also $\norm{X_L^{(0)}}\le\norm{X_L}\le1/2$.  Therefore
\[
\sigma_{\nu k+1}(X_L^{(0)})
\le
2\exp\!\left(-\frac{\pi^2k}{\log(16\gamma)}\right)
\le
2\exp\!\left(-\frac{\pi^2k}{2\log(8N|S|+4)}\right).
\]
\end{proof}

\subsection{The two blocks next to the time window}
Writing the components of $S$ as in \eqref{eq:component-form},
for each component $r$ set
$g_{W_r}(t):=\sin(2\pi W_rt)/(\pi t)$ for $t\ne0$, with
$g_{W_r}(0):=2W_r$.
Then
\begin{equation}
K_S(t)=\sum_{r=1}^{M}e^{i2\pi c_rt}g_{W_r}(t).
\label{eq:modulated-sincs}
\end{equation}
Since the intervals are disjoint,
$|S|=\sum_{r=1}^M 2W_r$,
and therefore $2W_r\le |S|$ for $1\le r\le M$. Hence
\eqref{eq:L1} implies
\begin{equation}
L_1\le\frac{1}{4W_r},
\qquad 1\le r\le M.
\label{eq:L1-Wr}
\end{equation}

We use the following one-band approximation estimate from \cite[Section~4.4]{Karnik21}.  The details are included in Appendix~\ref{app:chebyshev}.

\begin{lemma}[One-band near-boundary approximation]
\label{lem:scalar-near}
Let $0<W<1/2$, let $L_1\ge1$ satisfy $L_1\le(4W)^{-1}$, and let
\[
G[\ell,n]=\frac{\sin(2\pi W(\ell-n))}{\pi(\ell-n)},
\qquad -L_1\le\ell\le-1,
\quad 0\le n\le N-1.
\]
For every integer $k\ge1$, there exists a matrix $\widetilde G$ of rank at most $k$ such that
\begin{equation}
\norm{G-\widetilde G}_F^2
\le
\frac{5600}{\pi}\left(\frac{\pi}{48}\right)^{2k}.
\label{eq:scalar-near-bound}
\end{equation}
\end{lemma}

\begin{lemma}[Near blocks]
\label{lem:near-blocks}
For every integer $k\ge1$,
\begin{equation}
\sigma_{Mk+1}(X_L^{(1)})
\le
M\sqrt{\frac{5600}{\pi}}
\left(\frac{\pi}{48}\right)^k.
\label{eq:near-left}
\end{equation}
The same estimate holds for $X_L^{(2)}$.
\end{lemma}

\begin{proof}
For the $r$-th component, let $G_r$ be the matrix in Lemma~\ref{lem:scalar-near}
corresponding to  the centered interval $[-W_r,W_r]$.   Let $D_{r,\mathrm{row}}$ and $D_{r,\mathrm{col}}$ be diagonal unitary matrices with entries $e^{i2\pi c_r\ell}$ and $e^{i2\pi c_r n}$, respectively.  The contribution of the $r$-th component to $X_L^{(1)}$ is
$D_{r,\mathrm{row}}G_rD_{r,\mathrm{col}}^*$.
By \eqref{eq:L1-Wr} and Lemma~\ref{lem:scalar-near}, there is a rank-$k$ approximation $\widetilde G_r$ satisfying \eqref{eq:scalar-near-bound}.  Define
\[
\widetilde X_L^{(1)}
:=
\sum_{r=1}^{M}
D_{r,\mathrm{row}}\widetilde G_rD_{r,\mathrm{col}}^*.
\]
Then
$\rank(\widetilde X_L^{(1)})\le Mk$,
and, by unitary invariance and the triangle inequality,
\[
\norm{X_L^{(1)}-\widetilde X_L^{(1)}}_F
\le
\sum_{r=1}^{M}\norm{G_r-\widetilde G_r}_F
\le
M\sqrt{\frac{5600}{\pi}}
\left(\frac{\pi}{48}\right)^k.
\]
Since $\operatorname{rank}(\widetilde X_L^{(1)})\le Mk$, the
Eckart--Young theorem and $\|A\|\le \|A\|_F$
  give \eqref{eq:near-left}.

To compare the two near blocks, use the index changes
$\ell\mapsto N-1-\ell$ and $n\mapsto N-1-n$, which reverse the row and
column orders. If $\ell\in\{-L_1,\dots,-1\}$ and $n\in I_N$, then
\[
K_S\bigl((N-1-\ell)-(N-1-n)\bigr)
=K_S(n-\ell)
=\overline{K_S(\ell-n)}.
\]
Thus $X_L^{(2)}$ is obtained from $\overline{X_L^{(1)}}$ by row and column permutations.  Complex conjugation and permutations preserve singular values, so $X_L^{(2)}$ satisfies the same estimate.
\end{proof}

\subsection{Proof of the refined transition estimate}

\begin{proof}[Proof of Theorem~\ref{thm:second-main}]
The case $N=1$ is immediate, since $M_\eps(1,S)\le1$ while $2M(k_0+k_1)\ge4$. Assume $N\ge2$.  By \eqref{eq:three-block-gram} and Weyl's inequality for sums of Hermitian matrices,
\begin{equation}
\sigma_{2Mk_0+2Mk_1+1}(X_L)^2
\le
\sigma_{2Mk_0+1}(X_L^{(0)})^2
+
\sigma_{Mk_1+1}(X_L^{(1)})^2
+
\sigma_{Mk_1+1}(X_L^{(2)})^2.
\label{eq:weyl-combined}
\end{equation}
Since $\nu=2M$, by  Lemmas~\ref{lem:far-block} and \ref{lem:near-blocks} we have
\begin{equation}
\sigma_{2M(k_0+k_1)+1}(X_L)^2
\le
4\exp\!\left(-\frac{\pi^2k_0}{\log(8N|S|+4)}\right)
+
\frac{11200M^2}{\pi}
\left(\frac{\pi}{48}\right)^{2k_1}.
\label{eq:combined-sv}
\end{equation}
By the definition of $k_0$ in \eqref{eq:k0-main}, the first term is at most $4q_\eps/5$, and by the definition of $k_1$ in \eqref{eq:k1-main}, the second term is at most $q_\eps/5$.  Hence
\[
\sigma_{2M(k_0+k_1)+1}(X_L)^2\le q_\eps.
\]
Lemma~\ref{lem:transition} gives \eqref{eq:second-main}.  Finally, applying $\ceil{x}\le x+1$ to \eqref{eq:k0-main} and \eqref{eq:k1-main} gives \eqref{eq:second-main-simple}.
\end{proof}

\begin{remark}[Large frequency-set measure]
\label{rem:complement}
If $|S|\ge1/2$, apply Theorem~\ref{thm:second-main} to
$S^c=[-1/2,1/2]\setminus S$ whenever $0<|S^c|<1/2$. Since $S$ and
$S^c$ partition the frequency interval,
$B_{N,S^c}=I-B_{N,S}$.
Because the eigenvalues are listed in decreasing order,
$\lambda_k(N,S^c)=1-\lambda_{N-1-k}(N,S)$.
Thus the two transition counts are equal.  The number of components in Theorem~\ref{thm:second-main} is then the number of connected components of $S^c$ on $\T$.
\end{remark}

\section{Stability under perturbations of the frequency set}
\label{sec:stability}

The displacement-rank argument counts all boundary points of the frequency set $S$, no matter
how close they are. We now use standard perturbation
inequalities for Hermitian matrices to compare the spectra associated with
two nearby frequency sets; see, for example, \cite[Chapters~III and~VI]{Bhatia1997}.
For the prolate matrices considered here, the perturbation can be controlled
directly by the measure of the symmetric difference of the two frequency
sets. Combining this observation with Theorem~\ref{thm:first-main}  allows us to obtain bounds by replacing the frequency set with a simpler one having fewer boundary points.
For sets $A$ and $B$, we write $A\triangle B$ for their symmetric
difference and $|A\triangle B|$ for its measure.

\begin{proposition}[Stability under frequency perturbations]
\label{prop:frequency-set-stability}
Let $S,S_0\subset\T$ be measurable sets, set
$\delta:=|S\triangle S_0|$, and let $0<\eps<1/2$.
Then, for every $0<\tau<\eps$,
\begin{equation}
M_\eps(N,S)
\le
M_\tau(N,S_0)
+
\frac{N\delta}{\eps-\tau}.
\label{eq:set-perturbation-plunge-additive}
\end{equation}
Moreover, if $N\delta<\eps$, then
\begin{equation}
M_\eps(N,S)
\le
M_{\eps-N\delta}(N,S_0).
\label{eq:set-perturbation-plunge-shifted}
\end{equation}
\end{proposition}

\begin{proof}
Set $E_+=S\setminus S_0$ and $E_-=S_0\setminus S$.
Then
$B_{N,S}-B_{N,S_0}=B_{N,E_+}-B_{N,E_-}$.
Since $B_{N,E}\succeq0$ and
$\|B_{N,E}\|_1=\tr(B_{N,E})=N|E|$,
where $\|\cdot\|_1$ denotes the trace norm, we have
\begin{equation}
\|B_{N,S}-B_{N,S_0}\|_1
\le N|E_+|+N|E_-|=N\delta.
\label{ineq:l1-Ndelta}
\end{equation}
The Lidskii--Mirsky inequality for Hermitian matrices
\cite[Chapter~III]{Bhatia1997} gives
\begin{equation}
\sum_{k=0}^{N-1}
\bigl|\lambda_k(N,S)-\lambda_k(N,S_0)\bigr|
\le N\delta.
\label{eq:set-perturbation-eigs-l1}
\end{equation}

To prove \eqref{eq:set-perturbation-plunge-additive},  if $M_\eps(N,S)\le M_\tau(N,S_0)$, then the result is immediate.
Otherwise, 
let
$\mathcal I=\{k:\eps<\lambda_k(N,S)<1-\eps\}$, so that
$\#\mathcal I=M_\eps(N,S)$.
At most $M_\tau(N,S_0)$ of these indices satisfy
$\lambda_k(N,S_0)\in(\tau,1-\tau)$.
For each remaining index,
$|\lambda_k(N,S)-\lambda_k(N,S_0)| \ge\eps-\tau$.
Summing over these remaining indices and using
\eqref{eq:set-perturbation-eigs-l1} gives
\[
\bigl(M_\eps(N,S)-M_\tau(N,S_0)\bigr)
(\eps-\tau)
\le N\delta,
\]
which proves \eqref{eq:set-perturbation-plunge-additive}.

For the second estimate,
\eqref{eq:set-perturbation-eigs-l1} implies
$|\lambda_k(N,S)-\lambda_k(N,S_0)| \le N\delta$
for every $k$.
If $N\delta<\eps$, every index counted by
$M_\eps(N,S)$ therefore satisfies
\[
\eps-N\delta
<
\lambda_k(N,S_0)
<
1-(\eps-N\delta).
\]
This proves \eqref{eq:set-perturbation-plunge-shifted}.
\end{proof}

\subsection{A local estimate for the operator norm}
The upper bound in Lemma~\ref{lem:local-concentration} follows by
adapting the convolution argument of Donoho and
Logan~\cite{DonohoLogan1992}, using their cosine kernel on the torus.
The lower bound follows by testing with a normalized Dirichlet kernel.
We include the details to fix the normalization.

\begin{lemma}[Local concentration estimate]
\label{lem:local-concentration}
For every measurable $E\subset\T$,
\begin{equation}
\frac{4}{\pi^2}\mathcal D_N(E)
\le
\|B_{N,E}\|
\le
\omega_N(E).
\label{eq:local-concentration-norm}
\end{equation}
Consequently, for measurable $S,S_0\subset\T$,
\begin{equation}
\|B_{N,S}-B_{N,S_0}\|
\le
\omega_N(S\triangle S_0).
\label{eq:local-perturbation-norm}
\end{equation}
\end{lemma}

\begin{proof}
The case $N=1$ is immediate. Assume $N\ge2$, and set
$h=N^{-1}$ and $q=N-1$.

Following the cosine-kernel argument of Donoho and Logan
\cite{DonohoLogan1992}, let
\[
g(t)=e^{\pi iqt}\cos(\pi qt)
\mathbf 1_{[-h/2,h/2]}(t),
\qquad b=\|g\|_2^2.
\]
A direct calculation gives
$
b\ge\frac h2$. 
Moreover, for \(0\le n\le q\),
\[
\widehat g(n)-b
=
2\int_{-h/2}^{h/2}
\cos(\pi qt)\sin(\pi nt)\sin(\pi(q-n)t)\,dt
\ge0.
\] 

For every $p\in\mathcal P_{N-1}$, where
$\mathcal P_{N-1}:=\operatorname{span}\{e^{2\pi inx}:0\le n<N\}$,
we may therefore write $p=u*g$, with
$\|u\|_2\le b^{-1}\|p\|_2$.
Here, $\|\cdot\|_2$ denotes the $L^2(\mathbb T)$ norm, with
$\mathbb T$ identified with $[0,1)$, and $u*g$ denotes convolution on
$\mathbb T$. 

By Cauchy--Schwarz and Fubini's theorem,
\[
\begin{aligned}
\int_E|p(x)|^2\,dx
&\le b\sup_{y\in\T}
|E\cap[y-h/2,y+h/2]|\,
\|u\|_2^2\\
&\le \frac{\mathcal D_N(E)}{Nb}\|p\|_2^2\\
&\le 2\mathcal D_N(E)\|p\|_2^2.
\end{aligned}
\]
Hence $\|B_{N,E}\|\le2\mathcal D_N(E)$.
Together with
$\|B_{N,E}\|\le\operatorname{tr}(B_{N,E})=N|E|$,
this proves the upper bound.

For the lower bound, testing with the normalized
Dirichlet kernel
\[
p_x(\xi)=N^{-1/2}
\sum_{n=0}^{N-1}e^{2\pi in(\xi-x)}
\]
and using
$|p_x(\xi)|^2\ge4N/\pi^2$ for
$|\xi-x|\le1/(2N)$ gives
$\|B_{N,E}\| \ge\frac4{\pi^2}\mathcal D_N(E)$.

Finally, taking $E=S\triangle S_0$ and using
$-B_{N,E}\preceq B_{N,S}-B_{N,S_0} \preceq B_{N,E}$
gives \eqref{eq:local-perturbation-norm}.
\end{proof}
\begin{remark}
The local estimate controls the operator norm,
not the trace norm. In particular, when $G\subset J$ and
$S=J\setminus G$,
$\|B_{N,J}-B_{N,S}\|_1=N|G|$.
Thus the local estimate cannot simply replace
the total-measure term in
\eqref{eq:set-perturbation-plunge-additive}.
\end{remark}

\begin{proof}[Proof of Theorem~\ref{thm:effective-boundary}]
  After discarding zero-length intervals in the representation of \(S_0\) and merging intervals whose closures meet, the set \(S_0\) changes only on a set of measure zero. Hence \(B_{N,S_0}\), \(|S\triangle S_0|\), and \(\omega_N(S\triangle S_0)\) are unchanged, while the boundary count can only decrease. Theorem~\ref{thm:first-main} therefore gives
$M_t(N,S_0)\le\#\partial_\T S_0\,\Phi_N(t)$ for $0<t<1/2$,
for any finite union of intervals $S_0$; when $|S_0|=0$ or $1$,
the left-hand side is zero.

Choose $S_0$ with $|S\triangle S_0|\le\rho$ and
$\#\partial_\T S_0=\nu_{\mathrm{eff}}(S;\rho)$.
Combining the preceding bound with
\eqref{eq:set-perturbation-plunge-additive} proves
\eqref{eq:effective-boundary-additive}; taking $\tau=\eps/2$ gives
\eqref{eq:effective-boundary-half}.
If $N\rho<\eps$, \eqref{eq:set-perturbation-plunge-shifted} and
monotonicity in the threshold give
\[
M_\eps(N,S)
\le M_{\eps-N|S\triangle S_0|}(N,S_0)
\le M_{\eps-N\rho}(N,S_0)
\le\nu_{\mathrm{eff}}(S;\rho)\Phi_N(\eps-N\rho),
\]
which is \eqref{eq:effective-boundary-shifted}.

For \eqref{eq:effective-boundary-local}, choose $S_0$ attaining the
minimum in \eqref{eq:local-effective-boundary}.
Lemma~\ref{lem:local-concentration} and Weyl's inequality yield
$|\lambda_k(N,S)-\lambda_k(N,S_0)|\le\eta$ for every $k$.
Since $\eta<\eps$, it follows that
\[
M_\eps(N,S)
\le M_{\eps-\eta}(N,S_0)
\le\nu_{\mathrm{loc}}(N,S;\eta)\Phi_N(\eps-\eta).
\]
Finally, $\omega_N(E)\le N|E|$ makes every comparison set admissible
for $\nu_{\mathrm{eff}}(S;\rho)$ also admissible for
$\nu_{\mathrm{loc}}(N,S;N\rho)$, proving
\eqref{eq:local-global-boundary-comparison}.
\end{proof}

\begin{proof}[Proof of Corollary~\ref{cor:small-gaps}]
Taking $S_0=J$ gives $|S\triangle J|=g$ and
$\#\partial_\T J=2$, so
$\nu_{\mathrm{eff}}(S;g)\le2$ and
$\nu_{\mathrm{loc}}(N,S;\omega_N(G))\le2$.
The estimates \eqref{eq:small-gap-additive},
\eqref{eq:small-gap-shifted}, and \eqref{eq:small-gap-local}
follow from \eqref{eq:effective-boundary-half},
\eqref{eq:effective-boundary-shifted}, and
\eqref{eq:effective-boundary-local}, respectively.
Since $\Phi_N$ is nonincreasing on $(0,1/2)$,
$Ng\le\eps/2$ gives \eqref{eq:small-gap-half}.
The same conclusion follows when $\mathcal D_N(G)\le\eps/4$,
because then $\omega_N(G)\le2\mathcal D_N(G)\le\eps/2$.
\end{proof}

\begin{remark}[Exact and effective boundary counts]
\label{rem:exact-effective-boundary}
The displacement rank counts every boundary point of $S$. In contrast,
Theorem~\ref{thm:effective-boundary} allows us to compare $S$ with a
frequency set having fewer boundary points. The resulting bounds depend
either on the total measure of the change or on its concentration within
intervals of length $N^{-1}$.
\end{remark}

\begin{remark}
If the approximating set $S_0$ satisfies $0<|S_0|\leq 1/2$, the same argument may
be combined with Theorem~\ref{thm:second-main} instead of
Theorem~\ref{thm:first-main}.  This gives a corresponding version in which the
main logarithmic term depends on $N|S_0|$.
\end{remark}

\subsection{An example with many small gaps} 
The next result shows that the local estimate
can be substantially better than the total-measure
estimate, even after optimizing the choice of the
simpler frequency set.

\begin{proposition}
[Many small gaps]
\label{prop:local-global-separation}
Fix $0<\eps<1/2$ and let $N\ge2$ be even.
Put $a=\eps/8$ and $J=[0,1/2]$.
For each $0\le j<N/2$, choose an interval
$G_j\subset(j/N,(j+1)/N)$ with $|G_j|=a/N$ whose closure is contained in that cell.
Set $G=\bigcup_{j=0}^{N/2-1}G_j$ and $S_N=J\setminus G$.
Then $|G|=\eps/16$, $\#\partial_{\T}S_N=N+2$, and
\begin{equation}
M_\eps(N,S_N)
\le
2\Phi_N(\eps/2).
\label{eq:separated-family-local}
\end{equation}

On the other hand, for every $\rho\ge0$
and $0<\tau<\eps$,
\begin{equation}
\nu_{\mathrm{eff}}(S_N;\rho)\Phi_N(\tau)
+\frac{N\rho}{\eps-\tau}
\ge
\frac N{16}.
\label{eq:separated-family-old-additive}
\end{equation}
If $N\rho<\eps$, then also
\begin{equation}
\nu_{\mathrm{eff}}(S_N;\rho)
\Phi_N(\eps-N\rho)
>
\frac N2-8.
\label{eq:separated-family-old-shifted}
\end{equation}
Thus the local estimate gives an
$O_\eps(\log N)$ bound, whereas the total-measure
bounds, even after optimization, have right-hand
sides bounded below by quantities of order $N$.
\end{proposition}

\begin{proof}
There are $N/2$ gaps, each of length
$a/N$. Thus $|G|=a/2=\eps/16$,
and $S_N$ has $N/2+1$ components.

Every interval of length $N^{-1}$
intersects at most two of the grid cells
in sets of positive measure. Hence
$\mathcal D_N(G)\le2a$ and $\omega_N(G)\le4a=\eps/2$.
Corollary~\ref{cor:small-gaps}
therefore gives
\eqref{eq:separated-family-local}.

Now let $S_0$ be any finite union of
intervals satisfying
$|S_N\triangle S_0|\le\rho$ and $r:=\#\partial_{\T}S_0$.
At most $r$ of the $N/2$ cell interiors
contain a boundary point of $S_0$.
On every remaining cell,
$\mathbf 1_{S_0}$ is constant almost
everywhere.

Since $S_N$ occupies a proportion
$1-a$ of each cell, and $a<1/2$,
the error on every such cell is at
least $a/N$. Therefore
\begin{equation}
N\rho\ge a(N/2-r)_+
\ge a(N/2-r).
\label{eq:cell-counting-obstruction}
\end{equation}

Choose $S_0$ attaining
$r=\nu_{\mathrm{eff}}(S_N;\rho)$.
Since $\Phi_N(\tau)\ge1$, we obtain
\begin{align*}
r\Phi_N(\tau)
+\frac{N\rho}{\eps-\tau}
\ge
r+\frac{N\rho}{\eps}
 \ge
\frac{aN}{2\eps}
+\left(1-\frac a\eps\right)r
 \ge\frac N{16},
\end{align*}
which proves
\eqref{eq:separated-family-old-additive}.

If $N\rho<\eps$, then
\eqref{eq:cell-counting-obstruction}
gives
$r>\frac N2-\frac{\eps}{a} =\frac N2-8$.
Since $\Phi_N(\eps-N\rho)\ge1$,
this proves
\eqref{eq:separated-family-old-shifted}.
\end{proof}

The lower bounds in \eqref{eq:separated-family-old-additive} and
\eqref{eq:separated-family-old-shifted} concern the right-hand sides
of \eqref{eq:effective-boundary-additive} and
\eqref{eq:effective-boundary-shifted}, respectively. They are not
lower bounds for $M_\eps(N,S_N)$.

\section{Location of the plunge region and individual eigenvalues}
\label{sec:eigenvalue-decay}

The $1/2$-crossing estimate of
Zhu and Wakin~\cite[Theorem~3.3]{wakin2017approximating} supplies the
indices locating the \(1/2\) crossing for the individual eigenvalue bounds. Set
\begin{equation}
q_-:=\floor{N|S|}-\nu+1,
\qquad
q_+:=\ceil{N|S|}+\nu-2.
\label{eq:qpm}
\end{equation}
In our notation, their theorem gives integers $\iota_-$ and
$\iota_+$ satisfying
\begin{equation}
\floor{N|S|}-2M+2
\le \iota_-\le\iota_+
\le \ceil{N|S|}+2M-2,
\label{eq:ZW-iota}
\end{equation}
and
\begin{equation}
\lambda_{\iota_--1}\ge\frac12,
\qquad
\lambda_{\iota_+}\le\frac12,
\label{eq:ZW-crossing}
\end{equation}
whenever the corresponding indices belong to $\{0,\ldots,N-1\}$.

Monotonicity of the eigenvalues gives the following consequence.

\begin{proposition}[$1/2$ crossing near $N|S|$]
\label{prop:crossing}
Let $S$ be a union of $M$  frequency intervals and let $\nu=2M$.
Define $q_-$ and $q_+$ by \eqref{eq:qpm}. If
$q_-\in\{0,\ldots,N-1\}$, then $\lambda_{q_-}\ge\frac12$.
If $q_+\in\{0,\ldots,N-1\}$, then $\lambda_{q_+}\le\frac12$.
\end{proposition}

\begin{proof}
Since $\nu=2M$, \eqref{eq:ZW-iota} gives
\[
\iota_--1\ge \floor{N|S|}-\nu+1=q_-,
\qquad
\iota_+\le \ceil{N|S|}+\nu-2=q_+.
\]
By the monotonicity of the eigenvalue sequence and
\eqref{eq:ZW-crossing}, the result is immediate.
\end{proof}

 For a single interval,  Karnik, Romberg, and Davenport
\cite[Corollary~1]{Karnik21} used the $1/2$ crossing together with their
transition bound to obtain explicit estimates for the individual
eigenvalues. Applying the same argument to Theorem~\ref{thm:first-main}
and Proposition~\ref{prop:crossing} gives the following multiband version.

\begin{corollary}[Exponential eigenvalue decay]
\label{cor:eigenvalue-decay}
Assume the hypotheses of Theorem~\ref{thm:first-main}  and let $q_-,q_+$ be as defined in \eqref{eq:qpm}. Put
\begin{equation}
a_N:=\frac{\pi^2}{\nu\log(4N)}.
\label{eq:aN}
\end{equation}
If $q_-\ge0$, then for every $0\le k\le q_-$,
\begin{equation}
\lambda_k
\ge
1-
8\exp\!\left[-a_N\bigl(q_--k+1-\nu\bigr)\right].
\label{eq:eig-lower}
\end{equation}
If $q_+\le N-1$, then for every $q_+\le k\le N-1$,
\begin{equation}
\lambda_k
\le
8\exp\!\left[-a_N\bigl(k-q_++1-\nu\bigr)\right].
\label{eq:eig-upper}
\end{equation}
\end{corollary}

\begin{proof}
We prove \eqref{eq:eig-lower}; the proof of \eqref{eq:eig-upper} uses the same argument.  Fix $k\le q_-$ and set
\[
d:=q_--k+1,
\qquad
\eta:=8\exp\!\left[-a_N(d-\nu)\right].
\]
If $\eta\ge1/2$, then \eqref{eq:eig-lower} follows from Proposition~\ref{prop:crossing}.  Suppose $\eta<1/2$ and, for contradiction, that $\lambda_k<1-\eta$.  Since $\lambda_{q_-}\ge1/2>\eta$, all eigenvalues with indices $k,k+1,\dots,q_-$ lie in $(\eta,1-\eta)$.  Hence
$d\le M_\eta(N,S)$.
On the other hand, by Theorem~\ref{thm:first-main} and $1-\eta>1/2$, we have
\begin{align*}
M_\eta(N,S)
 \le
\nu\ceil{\frac{\log(4N)}{\pi^2}
\log\!\left(\frac{4}{\eta(1-\eta)}\right)} <
\nu\frac{\log(4N)}{\pi^2}\log\!\left(\frac{8}{\eta}\right)+\nu
=d,
\end{align*}
where the last identity follows from the definition of $\eta$.  This is a contradiction, so \eqref{eq:eig-lower} follows.
\end{proof}

The corresponding estimates for a single interval appear in
\cite[Corollary~2]{Karnik21}. Summing the multiband bounds in
Corollary~\ref{cor:eigenvalue-decay} gives the following result.

\begin{corollary}[Partial sums]
\label{cor:partial-sums}
Under the hypotheses of Corollary~\ref{cor:eigenvalue-decay}, if $1\le K\le q_-+1$, then
\begin{equation}
\sum_{k=0}^{K-1}(1-\lambda_k)
\le
\frac{8\exp\!\left[-a_N(q_--K+2-\nu)\right]}{1-e^{-a_N}}.
\label{eq:leading-sum}
\end{equation}
If $q_+\le K\le N-1$, then
\begin{equation}
\sum_{k=K}^{N-1}\lambda_k
\le
\frac{8\exp\!\left[-a_N(K-q_++1-\nu)\right]}{1-e^{-a_N}}.
\label{eq:trailing-sum}
\end{equation}
\end{corollary}

The proof follows from \eqref{eq:eig-lower} and \eqref{eq:eig-upper} by summing the corresponding geometric series.

\section{Effective dimension of sampled multiband signals}
\label{sec:signal-approximation}
We recall the standard singular-value characterization of optimal
finite-dimensional approximation;  see
\cite[Proposition~1]{zhu2023time}.
We then apply the eigenvalue estimates above to sampled multiband signals. 
Let $\mathcal H_S:=\{x\in\ell^2(\Z):B_Sx=x\}$ be the space of
$S$-bandlimited sequences, and define the restriction operator
\[
R_{N,S}:\mathcal H_S\longrightarrow\C^N,
\qquad
R_{N,S}x=(x[0],\ldots,x[N-1]).
\]
As before, we identify $\C^N$ with the sequences supported on $I_N$.
For a subspace $V\subset\C^N$, let $P_V$ denote the orthogonal projection
onto $V$.

Let $u_0,\ldots,u_{N-1}$ be corresponding orthonormal eigenvectors
of $B_{N,S}$. For $0\le K<N$, set
$V_K:=\operatorname{span}\{u_0,\ldots,u_{K-1}\}$, with $V_0:=\{0\}$.
Then
\begin{equation}\label{eq:optimal-signal-error}
\sup_{\substack{x\in\mathcal H_S\\ \|x\|_{\ell^2}\le1}}
\|R_{N,S}x-P_{V_K}R_{N,S}x\|_2
=
\sqrt{\lambda_K}, \end{equation}
and
 \begin{equation}\label{eq:optimal-signal-error-min}
\inf_{\substack{V\subset\C^N\\ \dim V=K}}
\sup_{\substack{x\in\mathcal H_S\\ \|x\|_{\ell^2}\le1}}
\|R_{N,S}x-P_VR_{N,S}x\|_2
=
\sqrt{\lambda_K}.
 \end{equation}
For \(y\in\mathbb C^N\), viewed as a sequence supported on \(I_N\), we have
$R_{N,S}^*y=B_Sy$. 
Hence \(R_{N,S}R_{N,S}^*=B_{N,S}\). It follows that the singular values of \(R_{N,S}\) are
\(\sqrt{\lambda_0},\ldots,\sqrt{\lambda_{N-1}}\), and the leading
left singular vectors span \(V_K\).

For $0<\delta<1$, define the approximation dimension
\begin{equation}
d_\delta(N,S)
:=
\min\left\{
K:
\begin{array}{l}
\text{there exists }V\subset\C^N,\ \dim V=K,\text{ such that}\\[1mm]
\displaystyle
\sup_{\substack{x\in\mathcal H_S\\\|x\|_{\ell^2}\le1}}
\|R_{N,S}x-P_VR_{N,S}x\|_2\le\delta
\end{array}
\right\}.
\label{eq:signal-effective-dimension}
\end{equation} Thus $d_\delta(N,S)$ is the smallest dimension giving uniform error
at most $\delta$ on the unit ball of $\mathcal H_S$. Clearly \(d_\delta(N,S)\le N\), since \(V=\mathbb C^N\) gives zero error.   
By \eqref{eq:optimal-signal-error}--\eqref{eq:optimal-signal-error-min},
\begin{equation}\label{eq:dimension-eigenvalue-count}
d_\delta(N,S)
=
\#\{k:\lambda_k(N,S)>\delta^2\}.
\end{equation}

Combining \eqref{eq:dimension-eigenvalue-count} with
Proposition~\ref{prop:crossing} and Theorem~\ref{thm:first-main}
gives the following bound for the approximation dimension.

\begin{theorem}[Non-asymptotic bound for the approximation dimension]
\label{thm:signal-effective-dimension}
Let $S\subset\T$ be a finite union of intervals with $0<|S|<1$, and let
$\nu=\#\partial_{\T}S$. Let $0<\delta<1/\sqrt{2}$. Then
\begin{equation}
\max\left\{0,\floor{N|S|}-\nu+2\right\}
\le
d_\delta(N,S)
\le
\min\left\{
N,\,
\ceil{N|S|}+\nu-2
+\nu\Phi_N(\delta^2)
\right\}.
\label{eq:signal-dimension-bound}
\end{equation}
\end{theorem}

\begin{proof}
Let $q_-,q_+$ be as in \eqref{eq:qpm}.
If $q_-\ge0$, Proposition~\ref{prop:crossing} gives
$\lambda_{q_-}\ge1/2>\delta^2$, so
$d_\delta(N,S)\ge q_-+1$ by \eqref{eq:dimension-eigenvalue-count}.
If $q_-<0$, the lower bound is trivial.

For the upper bound, suppose $q_+\le N-1$.
Every counted index $k\ge q_+$ satisfies
$\delta^2<\lambda_k\le1/2<1-\delta^2$, and hence is counted by
$M_{\delta^2}(N,S)$. Therefore Theorem~\ref{thm:first-main} gives
\[
d_\delta(N,S)
\le q_+ + M_{\delta^2}(N,S)
\le \ceil{N|S|}+\nu-2+\nu\Phi_N(\delta^2).
\]
Together with $d_\delta(N,S)\le N$, this proves the upper bound;
if $q_+>N-1$, the latter inequality suffices.
\end{proof}

Equation~\eqref{eq:optimal-signal-error} and
Corollary~\ref{cor:eigenvalue-decay} also give an explicit error bound
for the leading-eigenvector approximation.

\begin{corollary}[Exponential approximation error]
\label{cor:signal-exponential}
Assume the hypotheses of Corollary~\ref{cor:eigenvalue-decay}, and
suppose $q_+\le K\le N-1$. Then the space
$V_K=\operatorname{span}\{u_0,\ldots,u_{K-1}\}$ satisfies
\begin{equation}
\sup_{\substack{x\in\mathcal H_S\\\|x\|_{\ell^2}\le1}}
\|R_{N,S}x-P_{V_K}R_{N,S}x\|_2
\le
2\sqrt2\,
\exp\left[
-\frac{\pi^2}{2\nu\log(4N)}
\bigl(K-q_++1-\nu\bigr)
\right].
\label{eq:signal-exponential-error}
\end{equation}
\end{corollary}

\begin{proof}
By  \eqref{eq:optimal-signal-error}, the left-hand side of
\eqref{eq:signal-exponential-error} equals $\sqrt{\lambda_K}$. Applying
\eqref{eq:eig-upper} and taking square roots gives the result.
\end{proof}

The local estimate of Lemma~\ref{lem:local-concentration}
also gives a comparison of approximation dimensions when the
frequency set is changed.

\begin{corollary}[Approximation dimension under frequency changes]
\label{cor:signal-perturbation}
Let $S,S_0\subset\T$ be finite unions of intervals with
$0<|S|,|S_0|<1$, and set $E_+:=S\setminus S_0$ and
$\eta_+:=\omega_N(E_+)$.
If $0<\delta<1$ and $\eta_+<\delta^2$, then
\begin{equation}
d_\delta(N,S)
\le
d_{\sqrt{\delta^2-\eta_+}}(N,S_0).
\label{eq:signal-dimension-perturbation}
\end{equation}

If, in addition, $0<\delta<1/\sqrt{2}$ and
$\nu_0=\#\partial_{\T}S_0$, then
\begin{equation}
\begin{aligned}
d_\delta(N,S)
\le
\min\Bigl\{
N,\,
\ceil{N|S_0|}+\nu_0-2
+\nu_0\Phi_N(\delta^2-\eta_+)
\Bigr\}.
\end{aligned}
\label{eq:signal-dimension-local}
\end{equation}

In particular, if $S\subseteq J$, where
$J\subset\T$ is an interval with $0<|J|<1$,
then, for every $0<\delta<1$,
$d_\delta(N,S)\le d_\delta(N,J)$.
Consequently, for $0<\delta<1/\sqrt{2}$,
\begin{equation}
d_\delta(N,S)
\le
\min\left\{
N,\,
\ceil{N|J|}+2\Phi_N(\delta^2)
\right\}.
\label{eq:signal-small-gap-dimension}
\end{equation}
\end{corollary}

\begin{proof}
Positivity and Lemma~\ref{lem:local-concentration} give
\[
B_{N,S}
=B_{N,S\cap S_0}+B_{N,E_+}
\preceq B_{N,S_0}+B_{N,E_+}
\preceq B_{N,S_0}+\eta_+I.
\]
Thus the min--max principle yields
$\lambda_k(N,S)\le\lambda_k(N,S_0)+\eta_+$.
When $\eta_+<\delta^2$, every index with
$\lambda_k(N,S)>\delta^2$ consequently satisfies
$\lambda_k(N,S_0)>\delta^2-\eta_+$.
Equation~\eqref{eq:dimension-eigenvalue-count} gives
\eqref{eq:signal-dimension-perturbation}.
If also $\delta<1/\sqrt2$, apply
Theorem~\ref{thm:signal-effective-dimension} to $S_0$ at accuracy
$\sqrt{\delta^2-\eta_+}$ to obtain \eqref{eq:signal-dimension-local}.

For $S\subseteq J$, taking $S_0=J$ gives $E_+=\varnothing$ and
$\eta_+=0$. Hence $d_\delta(N,S)\le d_\delta(N,J)$, and
\eqref{eq:signal-small-gap-dimension} follows from
Theorem~\ref{thm:signal-effective-dimension} because $J$ has two
boundary points.
\end{proof}
For $S=J\setminus G$, the inclusion $S\subseteq J$ gives
\eqref{eq:signal-small-gap-dimension} without any condition on the
number or total length of the gaps. The plunge count, in contrast,
need not be monotone under inclusion. For nonnested sets,
\eqref{eq:signal-dimension-perturbation} controls the approximation
dimension through the local concentration of the added frequencies.

There is also a mean-square interpretation
\cite{Davenport2012,wakin2017approximating}.
If $Y\in\C^N$ has mean zero and covariance $B_{N,S}$, then the
Karhunen--Lo\`eve principle gives
\[
\min_{\dim V=K}\mathbb E\|Y-P_VY\|_2^2
=\mathbb E\|Y-P_{V_K}Y\|_2^2
=\sum_{k=K}^{N-1}\lambda_k.
\]
For $q_+\le K\le N-1$, Corollary~\ref{cor:partial-sums} bounds this
error by the right-hand side of~\eqref{eq:trailing-sum}.

\section{Continuous multiband limiting operators}
\label{sec:continuous}

We transfer the discrete estimates to the continuous multiband
operator by a step-function approximation. 

Let $\Omega\subset\R$ be a bounded union of $M$ pairwise disjoint
intervals. Define $K_\Omega(t):=\int_\Omega e^{i2\pi\xi t}\,d\xi$
and the integral operator $\mathcal T_\Omega:L^2([0,1])\to L^2([0,1])$ by
\begin{equation}
(\mathcal T_\Omega f)(x)
=
\int_0^1K_\Omega(x-y)f(y)\,dy.
\label{eq:continuous-op}
\end{equation}
Since $K_\Omega(-t)=\overline{K_\Omega(t)}$, the operator is
self-adjoint. It is obtained by extending a function on $[0,1]$
by zero, projecting onto frequencies in $\Omega$, and then restricting
back to $[0,1]$. Therefore $0\preceq\mathcal T_\Omega\preceq I$.
It is compact, and we denote its eigenvalues by
$1\ge\widetilde\lambda_0(\Omega)\ge\widetilde\lambda_1(\Omega)\ge\cdots\ge0$.
When the frequency set is clear, we write
$\widetilde\lambda_k=\widetilde\lambda_k(\Omega)$.
For $0<t<1/2$, set
$\widetilde M_t(\Omega):=\#\{k:t<\widetilde\lambda_k<1-t\}$.

We use the following discretization to connect the discrete and continuous problems.

\begin{lemma}[Discrete-to-continuous convergence]
\label{lem:nystrom} For all sufficiently large $N$, 
set $S_N:=\Omega/N=\{\xi/N:\xi\in\Omega\}$. 
Let $B_{N,S_N}$ be the corresponding discrete multiband prolate matrix.  Then, for every fixed $k$,
\begin{equation}
\lambda_k(N,S_N)\longrightarrow\widetilde\lambda_k
\qquad\text{as }N\to\infty.
\label{eq:eigen-convergence}
\end{equation}
\end{lemma}

\begin{proof} Since \(\Omega\) is bounded,  $S_N\subset[-1/2,1/2]$ for all sufficiently  large $N$.
For $n=0,\ldots,N-1$, let $J_n=[n/N,(n+1)/N)$, with $1$
included in $J_{N-1}$, and define the step-function map
$U_N:\C^N\to L^2([0,1])$  by $(U_Nv)(x)=\sqrt N\,v[n]$ for $x\in J_n$.
Then $U_N$ is an isometry.
Let $A_N$ be the integral operator with kernel
\[
K_N(x,y)
=
K_\Omega\!\left(\frac{m-n}{N}\right),
\qquad
x\in J_m,\quad y\in J_n.
\]
Since
$K_{S_N}(m-n) = \frac1N K_\Omega\!\left(\frac{m-n}{N}\right)$,
a direct calculation gives
$A_N=U_NB_{N,S_N}U_N^*$.
Hence the nonzero eigenvalues of $A_N$ are those of $B_{N,S_N}$.

For $x\in J_m$ and $y\in J_n$,
$\left| \frac{m-n}{N}-(x-y) \right| \le \frac2N$.
Since $K_\Omega$ is uniformly continuous on $[-1,1]$,
\[
\begin{aligned}
\|A_N-\mathcal T_\Omega\|_{\mathrm{HS}}^2
&=
\int_0^1\int_0^1
\left|K_N(x,y)-K_\Omega(x-y)\right|^2\,dx\,dy\\
&\le
\sup_{\substack{s,t\in[-1,1]\\ |s-t|\le 2/N}}
|K_\Omega(s)-K_\Omega(t)|^2
\longrightarrow 0.
\end{aligned}
\] 

Thus the convergence also holds in operator norm. The min--max principle
for compact self-adjoint operators now gives
$\lambda_k(N,S_N)\longrightarrow \widetilde\lambda_k(\Omega)$
for every fixed $k$.
\end{proof}

Passing Theorem~\ref{thm:second-main} to the limit gives the following explicit bound.

\begin{theorem}[Continuous multiband transition bound]
\label{thm:continuous}
Let $\Omega\subset\R$  have $M$ components, each of which is a bounded interval. For $\eps\in(0,1/2)$ set $q_\eps=\eps(1-\eps)$ and
\begin{align}
\kappa_0
&:=
\ceil{
\frac{\log(8|\Omega|+4)}{\pi^2}
\log\!\left(\frac{5}{q_\eps}\right)},
\label{eq:cont-k0}\\
\kappa_1
&:=
\ceil{
\frac{1}{2\log(48/\pi)}
\log\!\left(\frac{56000M^2}{\pi q_\eps}\right)}.
\label{eq:cont-k1}
\end{align}
Then
\begin{equation}
\widetilde M_\eps(\Omega)\le 2M(\kappa_0+\kappa_1).
\label{eq:continuous-bound}
\end{equation}
\end{theorem}

\begin{proof}
Set $S_N=\Omega/N$. By Lemma~\ref{lem:nystrom},
\[
\widetilde M_\eps(\Omega)
\le
\liminf_{N\to\infty} M_\eps(N,S_N).
\]
Indeed, every eigenvalue lying strictly in \((\varepsilon,1-\varepsilon)\) remains in this interval for all sufficiently large \(N\).  
For all sufficiently large $N$,
$|S_N|<\frac12$ and $N|S_N|=|\Omega|$,
and $S_N$ has the same $M$ components as $\Omega$. Therefore
Theorem~\ref{thm:second-main} gives
$M_\eps(N,S_N) \le 2M(\kappa_0+\kappa_1)$.
Taking the lower limit proves \eqref{eq:continuous-bound}.
\end{proof}


Applying Proposition~\ref{prop:frequency-set-stability} before the
limit also gives a small-gap estimate. Related componentwise and
symmetric-difference comparisons appear
in~\cite{Azimifard2026MeasureSensitive}. For a bounded interval
$J\subset\R$ and $0<t<1/2$, set
\[
\Psi_J(t)
:=
2\left[
\ceil{
\frac{\log(8|J|+4)}{\pi^2}
\log\!\left(\frac{5}{t(1-t)}\right)
}
+
\ceil{
\frac{1}{2\log(48/\pi)}
\log\!\left(\frac{56000}{\pi t(1-t)}\right)
}
\right].
\]
By Theorem~\ref{thm:continuous}, $\Psi_J(t)$ is an upper bound for the plunge count
of the operator associated with $J$.

\begin{corollary}[Small gaps in the continuous problem]\label{cor:continuous-small-gaps}
Let $J\subset\R$ be a bounded interval and let $G\subset J$ be a finite
union of pairwise disjoint intervals. Set
$
\Omega:=J\setminus G$ and
$g:=|G|.
$
Then, for every $0<\eps<1/2$,
\[
\widetilde M_\eps(\Omega)\le\Psi_J(\eps/2)+\frac{2g}{\eps}.
\]
If $g<\eps$, then $\widetilde M_\eps(\Omega)\le\Psi_J(\eps-g)$.
In particular, if $g\le\eps/2$, then
$\widetilde M_\eps(\Omega)\le\Psi_J(\eps/2)$.
\end{corollary}
 \begin{proof}
Set $S_N=\Omega/N$ and $J_N=J/N$.
For all sufficiently large $N$, both lie in $[-1/2,1/2]$,
$N|S_N\triangle J_N|=g$, and $N|J_N|=|J|$.
Proposition~\ref{prop:frequency-set-stability} and
Theorem~\ref{thm:second-main} give
\[
M_\eps(N,S_N)
\le M_{\eps/2}(N,J_N)+\frac{2g}{\eps}
\le\Psi_J(\eps/2)+\frac{2g}{\eps}.
\]
If $g<\eps$, \eqref{eq:set-perturbation-plunge-shifted} similarly gives
\[
M_\eps(N,S_N)\le M_{\eps-g}(N,J_N)\le\Psi_J(\eps-g).
\]
By Lemma~\ref{lem:nystrom},
$\widetilde M_\eps(\Omega)\le\liminf_{N\to\infty}M_\eps(N,S_N)$.
Passing to the lower limit proves both bounds.
When $g\le\eps/2$, monotonicity of $\Psi_J$ gives the last assertion.
\end{proof}

Thus, as in the discrete case, the estimate depends on the total length
of the gaps rather than on the number of gaps.


\section{Numerical examples}
\label{sec:numerics}
The examples address four questions: how close the two main bounds are
to computed transition counts, how the interval locations affect the
spectrum, how individual eigenvalues behave, and when the local
comparison improves the total-measure bound. These comparisons of bounds with
computed spectra do not establish optimality.

\subsection{Comparison of the transition bounds}
\label{ssec:transition}
\label{ssec:sharpness}
\label{ssec:crossover}
Figure~\ref{fig:spectra} compares three symmetric frequency sets with
$N=512$ and $|S|=0.2$, but with $\nu=2,4,8$ boundary points.
The transition is wider for the larger boundary counts in these
examples, in agreement with the dependence in Theorem~\ref{thm:first-main}.

\begin{figure}[htb]
\centering
\includegraphics[width=0.60\textwidth]{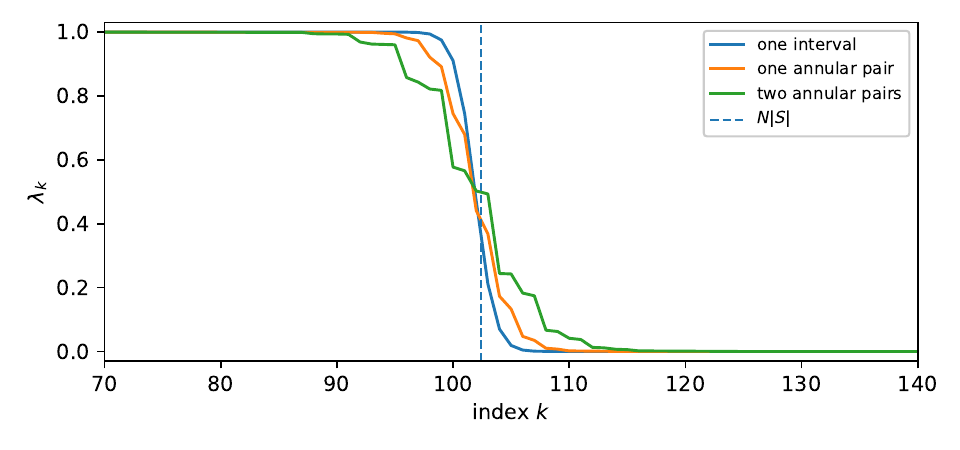}
\caption{Eigenvalues for $N=512$ and three frequency sets of measure
$0.2$: $[-0.1,0.1]$; $[-0.2,-0.1]\cup[0.1,0.2]$; and
$[-0.2,-0.15]\cup[-0.1,-0.05]\cup[0.05,0.1]\cup[0.15,0.2]$.
Their boundary counts are $2$, $4$, and $8$.
The vertical line marks $N|S|=102.4$.}
\label{fig:spectra}
\end{figure}
\FloatBarrier 

For a broader comparison, we use
\[
N\in\{64,128,256,512,1024,2048\},\qquad
|S|\in\{0.05,0.1,0.2,0.35\},\qquad M\in\{1,2,4,8\},
\]
with $\eps\in\{10^{-2},10^{-3},10^{-6}\}$.
Each set consists of $M$ intervals of length $|S|/M$, one centered in
each of the $M$ equal cells of $[-1/2,1/2)$.
Both bounds hold at all $288$ parameter combinations.
The ratio of~\eqref{eq:first-main} to the computed count ranges from
$1.17$ to $6.00$, with median $1.71$; the corresponding values
for~\eqref{eq:second-main-simple} are $1.65$, $8.75$, and $2.81$.
The first bound is smaller throughout this grid.

To determine when the time--bandwidth refinement is numerically
smaller, we extend the comparison to $
10^{-4}\le |S|\le \frac12$, and
$N\le65536$. For fixed $(N,M,\eps)$, write
\[
|S|^*=\sup\bigl\{s:\eqref{eq:second-main-simple}\text{ is smaller than }
\eqref{eq:first-main}\text{ when }|S|=s\bigr\},
\]
whenever such a value exists in the tested range.
Figure~\ref{fig:bounds} shows the comparison.
The largest observed crossover is $0.0140$, at $N=8192$, $M=1$,
and $\eps=10^{-6}$. Thus the refinement is useful here only for
small bandwidths; its crossover generally decreases with $M$.
For some choices with $\eps=10^{-2}$ there is no crossover in the
tested range. 

\begin{figure}[htb]
\centering
\includegraphics[width=0.99\textwidth]{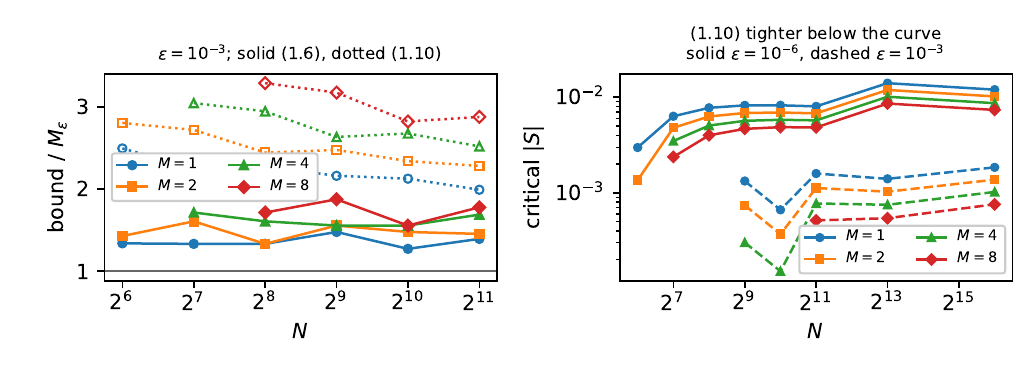}
\caption{Comparison of~\eqref{eq:first-main}
and~\eqref{eq:second-main-simple}. Left: median bound-to-count ratios
for $\eps=10^{-3}$, grouped by $N$ and $M$; solid markers denote
the first bound, and open markers with dotted lines the second.
Right: crossover bandwidth $|S|^*$; a missing marker means that no
crossover occurs in the tested range.}
\label{fig:bounds}
\end{figure}
\FloatBarrier 
\subsection{Dependence on the interval locations}
\label{ssec:location}
We fix $N=512$, $|S|=0.2$, and $M=4$, so $\nu=8$, and compare
$500$ random configurations. The four positive component widths are
chosen uniformly subject to having sum $0.2$; the four positive gap
lengths are chosen in the same way with sum $0.8$. Each configuration
is translated, if necessary, to avoid crossing the identified endpoints,
which leaves the eigenvalues unchanged by Lemma~\ref{lem:frequency-translation}.

For $\eps=10^{-6}$, the plunge counts range from $36$ to $59$,
although every configuration has the same $N$, $|S|$, and $\nu$.
All satisfy both bounds; the largest count is about $61\%$ of
\eqref{eq:first-main}. Figure~\ref{fig:location} compares the counts
with the smallest component width $w_{\min}$ and gap $g_{\min}$.
Their Spearman correlations with the count are $0.79$ and $0.38$,
respectively. The counts tend to decrease when $Nw_{\min}$ is near
$1$ or smaller.

For $\eps=10^{-6}$ and $10^{-3}$, the equal-band configuration has
a larger count than all $500$ random configurations; for $\eps=10^{-2}$,
two random configurations have larger counts. These observations
suggest finer dependence on the interval configuration, but do not
identify a maximizing set. The experiments vary the whole configuration,
not a prescribed small modification of one reference set.

\begin{figure}[htb]
\centering
\includegraphics[width=0.99\textwidth]{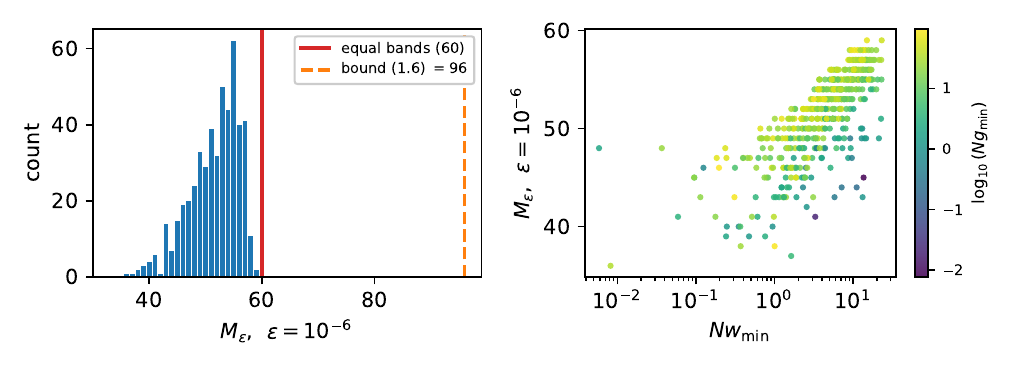}
\caption{The $500$ configurations with $N=512$, $|S|=0.2$, and $\nu=8$.
Left: distribution of $M_\eps$ for $\eps=10^{-6}$.
Right: $M_\eps$ against $Nw_{\min}$; the color of each point records
the smallest gap through $\log_{10}(Ng_{\min})$.}
\label{fig:location}
\end{figure}
 \subsection{Individual eigenvalues and the crossing index}
\label{ssec:asym}
\label{ssec:structural}
Corollary~\ref{cor:eigenvalue-decay} bounds $1-\lambda_k$ before the
transition and $\lambda_k$ after it with rate
$a_N=\pi^2/(\nu\log(4N))$.
The left panel of Figure~\ref{fig:decay} compares these bounds with
the eigenvalues. Across the cases considered,
fitting straight lines to $\log(1-\lambda_k)$ and $\log\lambda_k$
where the corresponding quantities lie between $10^{-13}$ and $10^{-1}$
gives rates between $2.69a_N$ and $3.68a_N$. These fits describe the
computed spectra, rather than establish an optimal rate.

The right panel compares the nonsymmetric set
$S=[-0.31,-0.22]\cup[0.05,0.11]\cup[0.27,0.33]$, for which
$|S|=0.21$ and $M=3$, with a symmetric set of the same measure and
component count. For $N=2^j$, $6\le j\le11$, and the three thresholds
above, the nonsymmetric counts satisfy both bounds and differ from
the symmetric counts by at most $2$.  This comparison concerns only these examples and does not imply that
changing the symmetry leaves the spectrum unchanged.  

\begin{figure}[htb]
\centering
\includegraphics[width=0.99\textwidth]{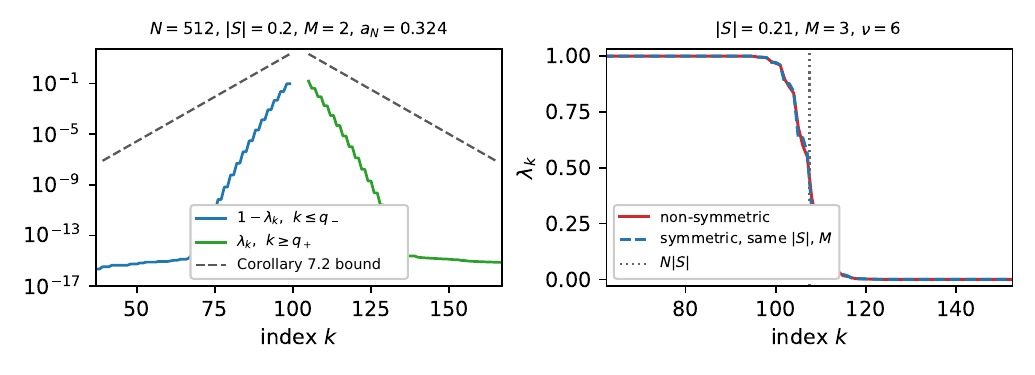}
\caption{Left: eigenvalue decay for $N=512$, $|S|=0.2$, and $M=2$,
with the bounds of Corollary~\ref{cor:eigenvalue-decay}; the flattening
near $10^{-16}$ reflects double-precision arithmetic.
Right: the nonsymmetric set specified in the text and the symmetric
comparison set with the same measure and component count.}
\label{fig:decay}
\end{figure}
\FloatBarrier 
For the equal-band grid, the crossing index
$k_{1/2}:=\min\{k:\lambda_k\le1/2\}$ lies in $[q_-,q_+]$ and satisfies
$|k_{1/2}-N|S||\le0.2\nu$ in the computations.
For the random configurations above, it takes only the values
$101,102,103$, within the theoretical window $[95,109]$. 
These examples do not imply that the distance of the crossing index
from $N|S|$ is generally proportional to $\nu$.  
\subsection{Many small gaps and the local estimate}
\label{ssec:local-gaps}
We illustrate Proposition~\ref{prop:local-global-separation} with
$\eps=10^{-2}$, $a=\eps/8$, and $J=[0,1/2]$.
For each even $N$ in Table~\ref{tab:local-gaps} and $0\le j<N/2$, set
\[
u_j=\frac14+\frac12\left((j+1)\sqrt2-\floor{(j+1)\sqrt2}\right),
\qquad
G_j=\left[\frac{j+u_j}{N},\frac{j+u_j+a}{N}\right].
\]
Put $G=\bigcup_jG_j$ and $S_N=J\setminus G$.
Each gap lies inside its cell, $|G|=\eps/16$, and
$\mathcal D_N(G)\le\eps/4$.
The total-measure bound using $S_0=J$, $\tau=\eps/2$, and the local
bound are, respectively,
\[
U_{\mathrm{total}}(N):=2\Phi_N(\eps/2)+\frac N8,
\qquad U_{\mathrm{local}}(N):=2\Phi_N(\eps/2).
\]
The first is a particular comparison, not the optimum over all sets.
Proposition~\ref{prop:local-global-separation} separately proves that
every right-hand side of~\eqref{eq:effective-boundary-additive} is at
least $N/16$ for this family.

\begin{table}[htb]
\centering
\caption{Dispersed gaps at $\eps=10^{-2}$. The fourth column is the
computed count; the last two are theoretical upper bounds.
$U_{\mathrm{total}}$ uses $S_0=J$ and $\tau=\eps/2$, without optimization.}
\label{tab:local-gaps}
\small
\begin{tabular}{rrrrrr}
\toprule
$N$ & Gaps & $N|G|$ & $M_\eps$ & $U_{\mathrm{total}}$ & $U_{\mathrm{local}}$\\
\midrule
64 & 32 & 0.04 & 6 & 16 & 8 \\
128 & 64 & 0.08 & 6 & 26 & 10 \\
256 & 128 & 0.16 & 7 & 42 & 10 \\
512 & 256 & 0.32 & 8 & 76 & 12 \\
1024 & 512 & 0.64 & 8 & 140 & 12 \\
\bottomrule
\end{tabular}
\end{table}
\FloatBarrier

Here $N|G|>\varepsilon$, so the condition required for the shifted
total-measure bound does not hold,  
but $\omega_N(G)\le\eps/2$ permits the local comparison.
At $N=1024$, the computed count is $8$ and the local upper bound is
$12$, whereas every right-hand side of
\eqref{eq:effective-boundary-additive} is at least $64$.
The latter is a lower bound on the competing bound expressions,
not on $M_\eps(N,S_N)$.

For this table, the matrices were assembled from exact formulas for the
interval integrals and computed in double precision. Direct integration
over $S_N$ was compared with $B_{N,J}-B_{N,G}$, and the computed
eigenvalues were well separated from the thresholds used to define the plunge region.

\section{Conclusion}\label{sec:conclusion}
The representation of multiband prolate matrices in terms of their
frequency endpoints gives explicit bounds for the eigenvalue transition
with logarithmic dependence on the threshold. The first bound is linear
in the number of frequency boundary points, while the second replaces
the main $\log N$ term by a logarithm involving the time--bandwidth
product $N|S|$. 

Comparing a frequency set with simpler sets gives bounds that also take
into account how small changes in the set are distributed. In particular,
for sets with many small gaps, these bounds can remain logarithmic in $N$
even when bounds based only on the total size of the gaps are much larger.
The estimates for individual eigenvalues, signal approximation, and the
continuous problem follow from these transition bounds.

The numerical examples indicate that the locations of the frequency
intervals contain information not captured by the number of boundary
points alone. Understanding this dependence and locating the $1/2$
crossing more precisely remain natural questions.

\appendix

\section{The Chebyshev estimate}
\label{app:chebyshev}

We include the estimate used in Lemma~\ref{lem:scalar-near}. The calculation follows \cite[Section~4.4 and Appendix~C]{Karnik21}.
For $W>0$, define
$g_W(t):=\frac{\sin(2\pi Wt)}{\pi t}
=\int_{-W}^{W}e^{i2\pi ft}\,df$.

\begin{lemma}[Derivative bound]
\label{lem:derivative-sinc}
For every integer $k\ge0$ and every $t\ne0$,
\begin{equation}
|g_W^{(k)}(t)|
\le
(2\pi W)^k
\min\left\{\frac{2W}{k+1},\frac{2}{\pi|t|}\right\}.
\label{eq:derivative-sinc}
\end{equation}
\end{lemma}

\begin{proof}
Differentiating the integral formula gives
$|g_W^{(k)}(t)|\le(2\pi W)^k\,2W/(k+1)$.
For $k\ge1$, one integration by parts bounds the endpoint terms
and the remaining integral each by $(2\pi W)^k/(\pi|t|)$.
For $k=0$, the second bound follows directly from the sine formula.
\end{proof}

We also use the standard Chebyshev interpolation estimate: if $h\in C^k[a,b]$ and $P_{k-1}$ is its degree-$(k-1)$ interpolant at the $k$ Chebyshev nodes, then
\begin{equation}
|h(t)-P_{k-1}(t)|
\le
\frac{(b-a)^k}{2^{2k-1}k!}
\max_{\xi\in[a,b]}|h^{(k)}(\xi)|.
\label{eq:cheb-standard}
\end{equation}
See \cite{SuliMayers2003}.

\begin{proof}[Proof of Lemma~\ref{lem:scalar-near}]
If $L_1=1$, take $\widetilde G=G$. Assume $L_1\ge2$.
For each $n\in I_N$, let $P_{k,n}$ be the degree-$(k-1)$
Chebyshev interpolant of $g_W(t-n)$ on $[-L_1,-1]$, and set
$\widetilde G[\ell,n]=P_{k,n}(\ell)$.
Its columns lie in the span of
$(\ell^j)_{-L_1\le\ell\le-1}$, $0\le j<k$, so
$\rank(\widetilde G)\le k$.
By \eqref{eq:cheb-standard} and Lemma~\ref{lem:derivative-sinc},
\begin{equation}
|G[\ell,n]-\widetilde G[\ell,n]|
\le
\frac{4\bigl(\frac\pi2W(L_1-1)\bigr)^k}{k!}
\min\left\{\frac{W}{k+1},\frac{1}{\pi(n+1)}\right\}.
\label{eq:pointwise-cheb}
\end{equation}
Set $N_1=\floor{(k+1)/(\pi W)}\ge1$.
After squaring \eqref{eq:pointwise-cheb}, split the sum in $n$
at $N_1$ and use
$\sum_{n=N_1}^{\infty}(n+1)^{-2}\le N_1^{-1}$ to obtain
\begin{equation}
\norm{G-\widetilde G}_F^2
\le
\frac{16L_1\bigl(\frac\pi2W(L_1-1)\bigr)^{2k}}{(k!)^2}
\left(
\frac{W^2N_1}{(k+1)^2}
+
\frac{1}{\pi^2N_1}
\right).
\label{eq:frob-mid}
\end{equation}
The assumptions $4WL_1\le1$ and $k\ge1$ imply
\[
L_1\bigl(4W(L_1-1)\bigr)^{2k}\le L_1-1,
\qquad
\frac{4(L_1-1)(k+1)}{\pi}
\le N_1\le\frac{k+1}{\pi W}.
\]
Consequently,
\[
(L_1-1)\left(
\frac{W^2N_1}{(k+1)^2}+\frac{1}{\pi^2N_1}
\right)
\le\frac{1}{2\pi(k+1)}.
\]
Substituting in \eqref{eq:frob-mid} gives
\[
\norm{G-\widetilde G}_F^2
\le
\frac{8}{\pi(k+1)(k!)^2}
\left(\frac{\pi}{8}\right)^{2k}.
\]
Finally, $(k+1)(k!)^2\ge6^{2k}/700$ for $k\ge1$:
the sequence $r_k=700(k+1)(k!)^2/36^k$ has successive ratio
$(k+1)(k+2)/36$, and its minimum is $r_5>1$.
Thus
\[
\norm{G-\widetilde G}_F^2
\le
\frac{5600}{\pi}
\left(\frac{\pi}{48}\right)^{2k}.
\]
\end{proof}

\paragraph{Declaration of AI assistance.}
During the preparation of this manuscript, the authors
used ChatGPT (OpenAI) for editorial assistance,
literature exploration, and suggestions concerning
mathematical arguments. The authors take full
responsibility for the mathematical statements,
proofs, references, and conclusions in the paper.
\newpage 

\printbibliography
\end{document}